\documentclass{article}
\usepackage[utf8]{inputenc}
\usepackage[margin=1in]{geometry}
\usepackage{amsmath} % assumes amsmath package installed
\usepackage{todonotes}
\usepackage[nolist]{acronym}
\usepackage{xcolor}
\usepackage{amsmath}
\usepackage{amssymb}
\usepackage{mathtools}
\usepackage{subcaption}
\usepackage{xcolor}
\usepackage{hyperref}
\hypersetup{colorlinks=true, linkcolor=black}
\usepackage{algorithm}
\usepackage{algpseudocode} 
\usepackage{mathrsfs}
\usepackage{cite}
\usepackage{adjustbox}
\usepackage{float}
\usepackage[edges]{forest}
\usetikzlibrary{trees} 
\usepackage{booktabs,multirow}

\newcommand{\rev}[1]{{\color{black} #1}}
\newcommand{\R}{\mathbb{R}} 
\newcommand{\K}{\mathbf{K}} 
 
\newcommand{\N}{\mathbb{N}}

\newcommand{\A}{\mathcal{A}}
\newcommand{\0}{\mathbf{0}}

\newcommand{\psd}{\mathbb{S}}

\newcommand{\vs}{\mathcal{V}}
\newcommand{\es}{\mathcal{E}}
\newcommand{\gs}{\mathcal{G}}

\newcommand{\cC}{\mathcal{C}}
\def\bB{{\mathbf{B}}}
\def\bG{{\mathbf{G}}}

\def\bm{{\mathbf{m}}}
\def\be{{\mathbf{e}}}
\def\a{{\boldsymbol{\alpha}}}

\def\b{{\boldsymbol{\beta}}}
\def\g{{\boldsymbol{\gamma}}}
\def\x{{\boldsymbol{x}}}
\def\y{{\boldsymbol{y}}}

\def\bq{{\boldsymbol{q}}}
\def\be{{\boldsymbol{e}}}
\def\bv{{\boldsymbol{v}}}

\def\cA{{\mathcal{A}}}
\def\cB{{\mathcal{B}}}

\def\G{{\mathcal{G}}}

\def\bS{{\mathbf{S}}}
\def\cs{{\mathrm{cs}}}
\def\var{\hbox{\rm{var}}}
\newcommand{\ud}{\mathrm{d}}

\newcommand{\mG}{\mathcal{G}}
\newcommand{\mC}{\mathcal{C}}

\newcommand{\mL}{\mathscr{L}}

\newcommand{\mI}{\mathcal{I}}
\newcommand{\mJ}{\mathcal{J}}

\newcommand{\mS}{{\mathbb S}}
\newcommand{\tr}[1]{\mbox{\upshape tr}\left(#1\right)}
\newcommand{\mE}{\mathcal{E}}

\DeclarePairedDelimiter{\abs}{\lvert}{\rvert}
\DeclarePairedDelimiter{\norm}{\lVert}{\rVert}

\DeclarePairedDelimiter{\rank}{\mathrm{rank}(}{)}
\DeclarePairedDelimiter{\diag}{\mathrm{diag}(}{)}

\DeclarePairedDelimiterX{\inp}[2]{\langle}{\rangle}{#1, #2}
\DeclarePairedDelimiter{\supp}{\mathrm{supp}(}{)}

\usepackage{amsthm}
\newtheorem{thm}{Theorem}[section]

\newtheorem{prop}[thm]{Proposition}
\newtheorem{cor}[thm]{Corollary}
\newtheorem{defn}{Definition}[section]

\newtheorem{exmp}{Example}[section]
\newtheorem{rem}{Remark}[section]

\renewcommand\footnotemark{}

\title{\Large \textbf{Sparse Polynomial Matrix Optimization}}

\author{Jared Miller$^1$, Jie Wang$^2$, Feng Guo$^3$
\thanks{$^1$J. Miller is with the Automatic Control Laboratory (IfA), Department of Information Technology and Electrical Engineering (D-ITET), ETH Z\"{u}rich, Physikstrasse 3, 8092, Z\"{u}rich, Switzerland (e-mail: jarmiller@control.ee.ethz.ch).}
\thanks{$^2$J. Wang is with State Key Laboratory of Mathematical Sciences, Academy of Mathematics and Systems Science, Chinese Academy of Sciences, Beijing, China (e-mail: wangjie212@amss.ac.cn).}
\thanks{$^3$ F. Guo is with the
School of Mathematical Sciences, Dalian University of Technology, Dalian 116024, Liaoning Province, China
(e-mail: fguo@dlut.edu.cn).}
\thanks{J. Miller was partially supported by the Swiss National Science Foundation Grant 51NF40\_180545 under NCCR Automation. J. Wang was supported by National Key R\&D Program of China under grant No. 2023YFA1009401, the Strategic Priority Research Program of the Chinese Academy of Sciences XDB0640000 \& XDB0640200, and the National Natural Science Foundation of China under grant No. 12201618 \& 12171324. F. Guo was supported by the Chinese National Natural Science Foundation under grant 12471478.}}

\begin{document}

\maketitle
% \thispagestyle{empty}
% \pagestyle{empty}

%%%%%%%%%%%%%%%%%%%%%%%%%%%%%%%%%%%%%%%%%%%%%%%%%%%%%%%%%%%%%%

\begin{abstract}
\label{sec:abstract}
\Iac{PMI} is a formula asserting that a polynomial matrix is positive semidefinite. 
\Ac{PMO} concerns minimizing the smallest eigenvalue of a symmetric polynomial matrix subject to a tuple of \acp{PMI}.
This work explores the use of sparsity methods in reducing the complexity of sum-of-squares based methods in verifying \acp{PMI} or solving \ac{PMO}. In the unconstrained setting, Newton polytopes can be employed to sparsify the monomial basis, resulting in smaller semidefinite programs. 
In the general setting, we show how to exploit different types of sparsity (term sparsity, correlative sparsity, matrix sparsity) encoded in polynomial matrices to derive sparse semidefinite programming relaxations for \ac{PMO}. 
For term sparsity, \rev{we show that the block structures of the term sparsity iterations with maximal chordal extensions converge to the one determined by PMI sign symmetries.}
For correlative sparsity, unlike the scalar case, we provide a counterexample showing that asymptotic convergence does not hold under the Archimedean condition and the running intersection property. By employing the theory of matrix-valued measures, we establish several results on detecting global optimality and retrieving optimal solutions under correlative sparsity.
The effectiveness of sparsity methods on reducing computational complexity is demonstrated on various examples of \ac{PMO}.
\end{abstract}
\section{Introduction}
\label{sec:introduction}
% This paper introduces to reduce the computational complexity of verifying \ac{PMI} constraints through the use of \ac{SOS} methods. It applies term sparsity techniques to perform this reduction, which performs decomposition based on the powers of the monomials present in the \ac{PMI}. 
This paper is concerned with improving the tractability of verifying of polynomial matrix inequalities (PMI) and of solving polynomial matrix optimization (PMO) problems.
\rev{A \ac{PMI} is a formula asserting that a polynomially-defined symmetric matrix is \ac{PSD} over some given set \cite{nie2011polynomial}.} \acp{PMI} are a generalization 
of (scalar) polynomial inequalities. 
Linear matrix inequalities are specific instances of \acp{PMI} restricted to degree one \cite{helton2007linear}, but \acp{PMI} may describe nonconvex sets. Applications of \acp{PMI} include finding stability regions of autoregressive linear systems \cite{ henrion2011inner}, sizing beams to minimize the cost of frame topologies under structural constraints \cite{tyburec2021global}, performing frequency-domain system identification of low-order linear systems under stability constraints \cite{abdalmoaty2023frequency}, and simplifying optimal control tasks in the case where the applied input is constrained to lie in a semidefinite-representable set \cite{miller2021analysis}. Refer to \cite{chesi2010lmi} for a survey of linear matrix inequality methods for solving polynomial optimization problems arising in control.

Polynomial optimization problems (POP) are programs where the objective and all constraints are defined by polynomials (forming a basic semialgebraic set). All POPs can be translated into equivalent programs with scalar linear objectives for which the constraint sets are described by polynomial inequality constraints. 
Checking nonnegativity of a polynomial over a basic semialgebraic set is generically an NP-hard problem, which extends (by reformulations in terms of membership oracles) to a generic NP-hardness of solving POPs \cite{murty1985some}. Methods to verify polynomial nonnegativity include satisfiability solvers  (e.g. dReal \cite{gao2013dreal}) and \ac{SOS} methods \cite{blekherman2012semidefinite}. 

If a real-valued polynomial with $n$ variables can be represented as an \ac{SOS} of other real-valued polynomials, then the original polynomial is certifiably nonnegative over the space $\R^n$. The set of SOS polynomials over $\R^n$ is strictly contained in the set of nonnegative polynomials over $\R^n$ \cite{blekherman2006there}. 
Checking if a polynomial is an \ac{SOS} over $\R^n$ can be accomplished by solving a finite-dimensional convex optimization problem, which can be numerically treated using \ac{SDP} \cite{blekherman2012semidefinite}. SOS methods can also be used in the constrained setting to verify nonnegativity over basic semialgebraic sets. If the constraint set is compact and its representing constraints satisfy a real-algebraic Archimedean structure, then every positive polynomial over the basic semialgebraic set can have its positivity verified using SOS methods \cite{putinar1993compact} (the polynomial degree needed to perform this verification is polynomial in the degree $d$ for fixed $n$ \cite{baldi2023effective}). 

These SOS methods for verifying polynomial nonnegativity have been extended to matrix \ac{SOS} for \acp{PMI} \cite{hol2005sum}, Hermitian \ac{SOS} for complex polynomials \cite{wang2022exploiting}, sums of Hermitian squares for noncommutative polynomials \cite{helton2002positive, cafuta2012noncommutative}. 
and flag SOS for graph density polynomials \cite{raymond2018symmetry}. 
The computational complexity of SOS implementation using SDP is determined by the size and multiplicity of the largest PSD matrix constraint \cite{lasserre2009moments}. In the context of a  $p \times p$ \ac{PMI} with $n$ variables where all polynomials are restricted to degree $2d$, the size of the largest PSD matrix (under the dense monomial basis) is $p\binom{n+d}{d}$. This matrix size (and the corresponding time to solve SDPs) suffer in a jointly polynomial manner as $p$, $n$, and $d$ grow. 
% SOS decompositions can also be verified through nonsymmetric interior point methods \cite{papp2019sum} without the use of SDPs, but the same joint polynomial increase in scaling will hold (with exponential scaling in dimension if an interpolatory polynomial basis is used for fast Hessian inversion).
Polynomial nonnegativity over sets admitting tractable Fourier analyses (e.g., ball, hypercube) can be accomplished in an optimization-free manner by checking a possibly exponential number of linear inequalities \cite{cristancho2024harmonic}. 
% Other methods for \ac{SOS} verification include using the Polya hierarchy of linear programs \cite{polya1928positive}, trust-region optimization over a low-rank manifold \cite{legat2023low}, and search methods using tame geometry \cite{aravanis2022polynomial}. 

Existing knowledge of the problem structure can be used in order to reduce the computational expenditure of SOS-based verification \cite{magron2023sparse,zheng2021chordal}. This decrease in complexity can be accomplished by a combination of reducing the size of monomial basis used in formulating the polynomial representation and identifying opportunities for decomposing PSD matrix constraints in the SOS representation. Three dominant (and interlinked) approaches include correlative sparsity, term sparsity, and symmetry. Correlative sparsity structure ties together variables that appear in the same constraint or the same monomial term in the candidate polynomial \cite{waki2006sums}. Term sparsity structure pays attention to the monomials that appear in polynomials (commutative \cite{wang2021tssos,wang3}, noncommutative \cite{wang2021exploiting}). 
The Newton polytope method in the unconstrained setting uses the exponents of monomials found in the candidate polynomial to determine which monomials can categorically be ruled out of an SOS representation \cite{Reznick_1978}. In the context of symmetry, enforcing that an SOS polynomial is additionally invariant/equivariant with respect to a given group action adds severe structural constraints to its SOS representation. The resultant SOS decomposition involves a block-diagonal PSD structure where the reduced monomial basis is comprised of primary and secondary invariants \cite{gatermann2004symmetry}. The term sparsity scheme in \cite{wang2021tssos} can be interpreted as a refinement of symmetry reduction with respect to the class of sign symmetries \cite{lofberg2009pre}. Algebraic structure can be exploited if POPs are posed over regions described by polynomial equality constraints (in addition to polynomial inequality constraints), either by using Gr\"{o}bner basis reduction over a quotient ring \cite{parrilo2005exploiting} or through the ideal-sparsity method by unfolding the separable equality constraints (e.g. $x(x-1)=0$). Multiple kinds of structure can be combined, such as the CS-TSSOS framework for correlative and term sparsity structures \cite{wang2022cs}. Term sparsity methods have also recently been used in sum-of-rational function optimization \cite{guo2024exploiting}. 
% The above exploitations of sparsity are all nonconservative (up to the degree-$d$ \ac{SOS} restriction). Computational efficiency ca 
Term sparsity methods can also be used in the derivation of non-SOS certificates of polynomial nonnegativity. The \ac{SONC} framework decomposes a polynomial as the sum of ‘circuit’ polynomials that are verifiably nonnegative due to their coefficients' satisfaction of the AM/GM inequality \cite{iliman2016amoebas}. The support and choice of these circuit polynomials are based on vertices of the Newton polytope (in the unconstrained setting). SONC verification therefore scales well when the candidate polynomial is sparse. SONC nonnegativity can be verified using relative entropy programs \cite{chandrasekaran2016relative} or second-order cone programs \cite{wang2020second,magron2023sonc}. 
% The more conservative but tractable dual \ac{SONC} cone enjoys a description as a polytope, and can therefore be searched over (at finite degree) through linear programming \cite{dressler2020global}. \ac{SONC} based verification is compatible with symmetry reduction over finite groups, which includes the class of sign symmetries \cite{moustrou2022symmetry}.

A basic \ac{PMO} problem is of form
\begin{equation}\label{pmo}
	\inf_{\x\in\R^n}\lambda_{\min}(F(\x))\quad \text{s.t. }\  G_1(\x)\succeq 0,\ldots,G_m(\x)\succeq 0, 
\end{equation}
where $F,G_1,\ldots,G_m$ are symmetric polynomial matrices, and $\lambda_{\min}(F(\x))$ denotes the smallest eigenvalue of $F(\x)$. When $F$ is a scalar polynomial (i.e., a $1\times1$ polynomial matrix), Problem \eqref{pmo} is also known as a polynomial SDP which was extensively studied in the literature \cite{kojima2003sums,henrion2006convergent,tran2024convergence,huang2024tightness}. When both $F$ and $G_1,\ldots,G_m$ are polynomial matrices, a matrix Moment-SOS hierarchy for solving \eqref{pmo} was provided in \cite{GW2023}.
SOS-simplifying structure can be extended to the matrix case. 
When $F$ is a polynomial matrix and $G_1,\ldots,G_m$ are polynomials, the work in \cite{zheng2023sum} utilizes the matrix (chordal) sparsity of $F$ to construct a sparse \ac{SOS} representation for $F$. Note that matrix sparsity is independent of the presence of specific monomials in the nonzero entries. When $F$ is a scalar polynomial, correlative sparsity was studied in \cite{KM2009,nie2024} and constraint matrix sparsity was considered in \cite{kim2011exploiting}. More recently, the work in \cite{handaterm} extends the term sparsity method to the case of scalar $F$ in the context of frame topology optimization. We summarize different types of PMI sparsity in Figure \ref{fig:PMI-sparsity}.
General sparsity methods for Problem \eqref{pmo} have not previously been considered in the literature. 
% to this paper is the frame topology optimization method , in which a scalar polynomial is verified to be nonnegative over a set defined by \acp{PMI}. The monomials of all elements in each \ac{PMI} constraint are aggregated together, independent of their location in the block matrix. 
This work fills this gap and studies sparsity reduction methods for \ac{PMO} in full generality where $F$ and $G_1,\ldots,G_m$ are all polynomial matrices. 

\begin{figure}[t!]
\centering
\tikzset{
    my node/.style={
        font=\small,
        rectangle,
        draw=#1!75,
        align=justify,
    }
}
\forestset{
    % 设置tree格式
    my tree style/.style={
        for tree={grow=south,    % 从左至右
            parent anchor=south, % 
            child anchor=north,  
        %设置parent node和child node的颜色、宽度等
        where level=0{my node=black}{},
        where level=1{my node=black}{},
        where level=2{my node=black}{},
            l sep=1.5em,
            forked edge,                %
            fork sep=1em,               %
            edge={draw=black!50, thick},                
            if n children=3{for children={
                    if n=2{calign with current}{}}
            }{},
            tier/.option=level,
        }
    }
}
% \scalebox{0.65}{
% 逆序
    \begin{forest}
      my tree style
[PMI sparsity
    [basis reduction]
    [term sparsity/PMI sign symmetry]
        [matrix sparsity
        [objective matrix sparsity]
        [constraint matrix sparsity]
        ]
    [correlative sparsity]
]
\end{forest}
% }
\centering
\caption{Different types of PMI sparsity.}
\label{fig:PMI-sparsity}
\end{figure}

This work makes the following contributions:
\begin{enumerate} 
\item In the unconstrained setting, we propose a method based on Newton polytopes for reducing monomial bases. 
\item We provide an iterative procedure to exploit term sparsity, which yields a bilevel hierarchy of sparse Moment-SOS relaxations for Problem \eqref{pmo}. \rev{In addition, we generalize the notion of (usual) sign symmetries for polynomials to the notion of PMI sign symmetries for polynomial matrices. It turns out that for \ac{PMO}, the block structures of the term sparsity iterations with maximal chordal extensions converge to the one determined by PMI sign symmetries of Problem \eqref{pmo}.} 
% Surprisingly, it turns out that the block structures produced by the iterative procedure with block closure do not necessarily converge to the one determined by sign symmetries of Problem \eqref{pmo}, which is dramatically different from the scalar case.
\item We show how to exploit correlative sparsity for \eqref{pmo}. When $F$ is a scalar polynomial, it is known that asymptotic convergence of the correlatively sparse relaxations holds under the Archimedean condition and the running intersection property. However, when $F$ is a polynomial matrix,
we give a counterexample showing that asymptotic convergence does not hold under similar conditions. Moreover, by employing the theory of matrix-valued measures, we establish several results on detecting global optimality and retrieving optimal solutions in the correlatively sparse setting.
\item Decomposition methods based on the matrix sparsity structure of the objective matrix (extending the work of \cite{zheng2023sum}) or \ac{PMI} constraints are also provided.
\item Extensive numerical experiments are performed, which demonstrate the efficacy of our methods. 
\end{enumerate}

The rest of this paper is organized as follows. 
Section \ref{sec:preliminaries} introduces preliminaries such as notation, matrix \ac{SOS}, matrix-valued measures, and graph theory. Section \ref{sec:uncons} applies the Newton polytope and term sparsity methods towards unconstrained \acp{PMI}. Section \ref{sec:cons} concerns the term sparsity method for constrained \ac{PMO}. Section \ref{sec:correlative} investigates correlative sparsity for \ac{PMO}. 
Section \ref{cms} considers matrix sparsity methods for \ac{PMO}.
Section \ref{sec:examples} demonstrates the effectiveness of these methods via various numerical examples. Section \ref{sec:conclusion} concludes the paper.

% Section \ref{sec:preliminaries} will review preliminaries such as notation, notions of stability for linear systems, and \ac{SOS} proofs of polynomial nonnegativity. Section \ref{sec:full_method} will present 
% The paper is concluded in Section \ref{sec:conclusion}.
\section{Preliminaries}
\label{sec:preliminaries}

% \subsection{Acronyms/Initialisms}
\begin{acronym}
\acro{BSA}{basic semialgebraic}
% \acro{DDC}{Data Driven Control}

\acro{CSP}{correlative sparsity pattern}

% \acro{GAS}{Globally Asymptotically Stable}

% \acro{CSP}{Correlative Sparsity Pattern}

\acro{LMI}{linear matrix inequality}
\acroplural{LMI}[LMIs]{linear matrix inequalities}
\acroindefinite{LMI}{an}{a}

% \acro{LQR}{Linear Quadratic Regulator}
% \acroplural{LMI}[LMIs]{Linear Matrix Inequalities}
% \acroindefinite{LQR}{an}{a}

\acro{LP}{linear program}
\acroindefinite{LP}{an}{a}
% \acro{OCP}{Optimal Control Problem}

% \acro{ODE}{Ordinary Differential Equation}

\acro{PMI}{polynomial matrix inequality}

\acro{POP}{polynomial optimization problem}

\acro{PMO}{polynomial matrix optimization}

\acro{PSD}{positive semidefinite}

\acro{SDP}{semidefinite program}
\acroindefinite{SDP}{an}{a}

\acro{SONC}{sum of nonnegative circuit}

\acro{SOS}{sum of squares}
\acroindefinite{SOS}{an}{a}

\acro{TSP}{term sparsity pattern}

\acro{WSOS}{weighted sum of squares}

\end{acronym}

\subsection{Notation}
The $n$-dimensional real Euclidean vector space is $\R^n$. The set of natural numbers is $\N$, and the $n$-dimensional set of multi-indices is $\N^n$. The symbol $[n]$ denotes the set $\{1,\ldots,n\}$. For a set $\A$, its cardinality is denoted by $|\A|$. The degree of a multi-index $\a \in \N^n$ is $\abs{\a} \coloneqq \max_{i \in [n]} \alpha_i$.
The associated monomial to an $n$-dimensional indeterminate $\x\coloneqq(x_1,\ldots,x_n)$ and a multi-index $\a \in \N^n$ is $\x^\a = \prod_{i=1}^n x_i^{\alpha_i}$. Letting $\A \subset \N^n$ be a finite-cardinality set of multi-indices and $\{c_\a\}_{\a \in \A}$ be an associated set of real numbers, the polynomial formed by $\A$ and $\{c_\a\}_{\a \in \A}$ is $f(\x) = \sum_{\a \in \A} c_\a \x^\a$. The set of polynomials with real-valued coefficients is $\R[\x]$. The \emph{support} of a polynomial $f \in \R[\x]$, denoted by $\supp{f}$, is the set of multi-indices $\a$ such that $c_\a \neq 0$.
The degree of a polynomial $f \in \R[\x]$ is $\deg{f} = \max_{\a \in \supp{f}} \abs{\a}$.
The set of polynomials of degree at most $d$ is denoted by $\R[\x]_{d}$. Let $\circ$ denote the Hadamard (entrywise) product.

The transpose of a matrix $A$ is denoted by $A^{\intercal}$. The $p$-dimensional identity matrix is $I_p$. The Kronecker product between two matrices $C$ and $D$ is $C \otimes D$. 
The set of $p\times p$ symmetric \rev{and real-valued} matrices is $\psd^p$. The subset of $p\times p$ \ac{PSD} matrices is $\psd^p_+$. Membership in the \ac{PSD} set $Z \in \psd^p_+$ will also be denoted as $\rev{Z} \succeq 0$, and the Loewner partial ordering will be used as $\rev{Z}_1 \succeq \rev{Z}_2 \Leftrightarrow \rev{Z}_1 - \rev{Z}_2 \succeq 0$. 
% The operator $\diag{a_1, a_2, \ldots}$ forms a (possibly rectangular) diagonal matrix with main elements $a_1, a_2, \ldots$ (and all other elements are zero). 

% The duality paring between two symmetric matrices $A, B \in \psd^n$ is $\inp{A}{B}_{\psd^n} = \sum_{ij} A_{ij} B_{ij}.$

The set of polynomial matrices of dimension $p \times q$ is $\R[\x]^{p \times q}$. Given a polynomial matrix $\rev{Z}(\x)=[\rev{Z}(\x)_{ij}]\in\R[\x]^{p \times q}$, we can write it as $\rev{Z}(\x)=\sum_{\a}\rev{Z}_{\a}\x^\a$, $\rev{Z}_{\a}\in\R^{p \times q}$. The \emph{support} of $\rev{Z}(\x)$, denoted by $\supp{\rev{Z}}$, is the set of multi-indices $\a$ such that $\rev{Z}_\a \neq 0$.
The degree of $\rev{Z}(\x)$ is $\deg{\rev{Z}} = \max_{i \in [p], \ j \in [q]} \deg{\rev{Z}(\x)_{ij}}$.
The set of symmetric polynomial matrices of dimension $p$ is $\psd^p[\x]$. The set of $p$-dimensional symmetric polynomial matrices of degree at most $d$ is denoted by $\psd^p[\x]_{d}$.

\subsection{Sum of Squares Polynomials}

This subsection will review \ac{SOS} methods for verifying nonnegativity of polynomials over $\R^n$ and over constrained \rev{semialgebraic} sets.

\subsubsection{Unconstrained}
A polynomial $f(\x) \in \R[\x]_{d}$ is nonnegative over $\R^n$ if $f(\x) \geq 0$ for any $\x \in \R^n$. One method to verify polynomial nonnegativity is to use \ac{SOS} certificates.
A polynomial $f(\x) \in \R[\x]$ is \ac{SOS} if there exist polynomials $g_1,\ldots,g_s\in \R[\x]$ such that $f(\x) = \sum_{i=1}^s g_i(\x)^2$. 
% Noting that $q_i(\x)$ takes on real values over $\R^n$ for each $i \in [s]$, it holds that each $q_i(\x)^2$ is nonnegative over $\R^n$ and thus $p(\x)$ is nonnegative over $\R^n$. 
The set of \ac{SOS} polynomials, denoted by $\Sigma[\x]$, is contained inside the set of nonnegative polynomials over $\R^n$. 
% This containment is tight (\ac{SOS}=nonnegative) only in the cases where $n=1$, $d=2$, or when $n=2$ and $d=4$ \cite{hilbert1888darstellung}; otherwise the set of \ac{SOS} polynomials is strictly contained in the set of nonnegative polynomials (and the volume of the set of SOS polynomials shrinks as compared to the corresponding volume of nonnegative polynomials as $n$ and $d$ increase \cite{blekherman2006there}). 

% Noting that NP-hard combinatorial optimization problems such as 3-SAT, stable sets, and max-cut can be represented as \acp{POP}, it holds that verification of polynomial nonnegativity is generally NP-hard \cite{murty1985some}.

A polynomial matrix $\rev{Z}(\x) \in \psd^p[\x]$ is positive semidefinite (PSD) over $\R^n$ if 
\begin{equation}
    \forall \x \in \R^n: \qquad \rev{Z}(\x) \succeq 0. \label{eq:pmi_uncons}
\end{equation} 
The statement in \eqref{eq:pmi_uncons} is \iac{PMI} in $\x$ over the unconstrained region $\R^n$.
The matrix $\rev{Z}(\x)$ is called an \ac{SOS} matrix if there exists another polynomial matrix $R(\x) \in \R[\x]^{s \times p}$ such that $\rev{Z}(\x) = R(\x)^{\intercal} R(\x)$.
The set of $n$-dimensional SOS matrices is denoted by $\Sigma^n[\x]$.
Note that the set of \ac{SOS} (scalar) polynomials $\Sigma[\x]$ is equivalent to the set $\Sigma^1[\x]$ (the $n=1$ case of \ac{SOS} matrices). 

Checking if a polynomial matrix $\rev{Z}(\x) \in \psd^p[\x]_{2d}$ is inside the set $\Sigma^p[\x]$ can be accomplished by solving an SDP. Letting $\bm_d(\x)$ be the $\binom{n+d}{d}$-vector of monomials in $\x$ up to degree $d$, the matrix $\rev{Z}(\x)$ is an \ac{SOS} matrix if there exists a \ac{PSD} matrix $Q \in \psd_+^{p \binom{n+d}{d}}$ such that
(Lemma 1 of \cite{scherer2006matrix})
\begin{equation}
    \rev{Z}(\x) = \left(\bm_d(\x) \otimes I_p\right)^{\intercal} Q (\bm_d(\x) \otimes I_p). \label{eq:sos_pmi_uncons}
\end{equation}
Problem \eqref{eq:sos_pmi_uncons} is an \ac{SDP} with respect to the \ac{PSD} matrix constraint $Q \in \psd_+^{p \binom{n+d}{d}}$ and the affine equality (coefficient matching) constraints in \eqref{eq:sos_pmi_uncons}.
Note that \eqref{eq:sos_pmi_uncons} can also be written as 
\begin{equation}
    \rev{Z}(\x) = \left(I_p\otimes \bm_d(\x) \right)^{\intercal} \tilde{Q} (I_p\otimes \bm_d(\x)), \label{eq:sos_pmi_uncons2}
\end{equation}
where $\tilde{Q}$ is obtained from $Q$ by certain row and column permutations.

Verification of the \ac{PMI} in \eqref{eq:pmi_uncons} can also be accomplished through \textit{scalarization}. The scalarization approach introduces a new tuple of variables $\y \in \R^p$ to form the equivalent problem:
\begin{equation}
    \forall (\x, \y) \in \R^{n+p}: \qquad \y^{\intercal} \rev{Z}(\x) \y \geq 0. \label{eq:pmi_uncons_scalarize}
\end{equation} 
The \ac{SOS}-restriction of the scalarized problem \eqref{eq:pmi_uncons_scalarize} is 
\begin{equation}
    \y^{\intercal} \rev{Z}(\x) \y \in \Sigma[\x, \y]. \label{eq:pmi_uncons_scalarize_sos}
\end{equation} 

\subsubsection{Constrained}
Let $\bG:=\{G_k(\x)\}_{k=1}^{m}$ be a set of symmetric polynomial matrices with $G_k(\x) \in \psd^{q_k}[\x]$ which defines a basic semialgebraic set as
\begin{align}
    \K \coloneqq \{\x \in \R^n \mid G_k(\x) \succeq 0, \quad \forall k \in [m]\}. \label{eq:bsa_set}
\end{align}
A polynomial matrix $\rev{Z}(\x) \in \psd^p[\x]$ satisfies \iac{PMI} over the constrained region $\K$ if
\begin{align}
    \forall \x \in \K: \qquad \rev{Z}(\x) \succeq 0. \label{eq:pmi_cons}
\end{align}
The following exposition of \ac{SOS} verification of \acp{PMI} originates from \cite{scherer2006matrix}.

Let $C$ be a a $pq\times pq$ block matrix  described as
\begin{align}
    C = \begin{bmatrix}
        C_{11} & C_{12} & \hdots & C_{1p} \\
        C_{21} & C_{22} & \hdots & C_{2p} \\
        \vdots & \vdots & \ddots & \vdots \\
        C_{p1} & C_{p2} & \hdots & C_{pp}
    \end{bmatrix} \label{eq:c_block}
\end{align}
in which each block $C_{ij}$ is of size $q\times q$.
The $p$-product between $C$ and a $q\times q$ matrix $D$ is defined by
\begin{align}
    \inp{C}{D}_p\coloneqq\begin{bmatrix}
        \text{tr}(C_{11}^{\intercal}D) & \text{tr}(C_{12}^{\intercal}D) & \hdots & \text{tr}(C_{1p}^{\intercal}D) \\
        \text{tr}(C_{21}^{\intercal}D) & \text{tr}(C_{22}^{\intercal}D) & \hdots & \text{tr}(C_{2p}^{\intercal}D) \\
        \vdots & \vdots & \ddots & \vdots \\
        \text{tr}(C_{p1}^{\intercal}D) & \text{tr}(C_{p2}^{\intercal}D) & \hdots & \text{tr}(C_{pp}^{\intercal}D)
    \end{bmatrix}.
\end{align}
% Define the operator $\text{tr}_p: \R^{n p \times n p} \rightarrow \R^{n \times n}$ as the following operation (w.r.t. a matrix $C$ from \eqref{eq:c_block}):
% % define $\inp{\cdot}{\cdot}_p: (\R^{n p \times n p})^2 \rightarrow \R^{p \times p}$ as the following operation \cite[page 192]{scherer2006matrix}:
% \begin{align}
%     \text{tr}_p(C) =   \begin{bmatrix}
%         \text{tr}(C_{11}) & \text{tr}(C_{12}) & \hdots & \text{tr}(C_{1p}) \\
%         \text{tr}(C_{21}) & \text{tr}(C_{22}) & \hdots & \text{tr}(C_{2p}) \\
%         \vdots & \vdots & \ddots & \hdots \\
%         \text{tr}(C_{p1}) & \text{tr}(C_{p2}) & \hdots & \text{tr}(C_{pp})
%     \end{bmatrix}.
% \end{align}
% A pairing $\inp{\cdot}{\cdot}_p$ can be created from matrices $C \in \R^{np \times np}$ and $D \in \R^{n \times n}$ as
% \begin{align}
%     \inp{D}{E}_p = \text{tr}_p(D^{\intercal} (I_p \otimes E)).
% \end{align}
A sufficient \ac{SOS} condition for \eqref{eq:pmi_cons} to hold is that there exist \ac{SOS} matrices $S_0\in \Sigma^p[\x], S_k\in\Sigma^{pq_k}[\x], k\in[m]$ such that 
\begin{equation}\label{eq:suff_sos}
    \rev{Z}(\x) = S_0(\x) + \sum_{k=1}^{m} \langle S_k(\x),G_k(\x)\rangle_p.
\end{equation}

The dimension-$p$ set of \ac{WSOS} polynomial matrices $\Sigma^p[\bG]$ is the
subset of $\psd^p[\x]$ that have a representation in \eqref{eq:suff_sos} (also
referred to the quadratic module formed by $\bG$). The degree-$\leq 2d$ set of
\ac{WSOS} polynomials $\Sigma^p[\bG]_{2d}$ is the subset of $\Sigma^p[\bG]$ where
$\deg S_0 \leq 2d$ and $\deg \langle S_k(\x),G_k(\x)\rangle_p\leq 2d$ for all $k\in[m]$.

%The set $\K$ from \eqref{eq:bsa_set} as represented 
The quadratic module $\Sigma^p[\bG]$ formed by $\bG$ satisfies the 
\textit{Archimedean} property if
there exists an $a \geq 0$ such that the scalar polynomial $a - \norm{\x}_2^2$ is a
member of $\Sigma^1[\bG]$. If $\Sigma^p[\bG]$ is Archimedean, then $\K$ is compact, but not necessarily vice versa
%every compact \ac{BSA} set is Archimedean 
\cite{cimpric2011quadratic}. If there exists an $\epsilon > 0$ such
that $\rev{Z}(\x) \succeq \epsilon I_p$ over $\K$ and $\Sigma^p[\bG]$ is Archimedean,
then Scherer and Hol's Positivestellensatz \cite[Corollary 1]{scherer2006matrix}
states that $\rev{Z}(\x)$ will have a representation in
\eqref{eq:suff_sos}. However, the degrees of the $S$ terms needed to form the
\eqref{eq:suff_sos} certification may be exponential in $n$ and $\deg \rev{Z}$ (even
for the $p=1$ case) \cite{nie2007complexity}.

\subsection{Matrix-Valued Measures and Moment Matrices}
We recall some basics about the theory of matrix-valued measures and the
associated moment matrices, which will be used to establish conditions for detecting global optimality and extracting optimal solutions from the Moment-\ac{SOS} relaxations of PMO problems. 

For the set $\K$ in \eqref{eq:bsa_set}, let
$B(\K)$ be the smallest $\sigma$-algebra generated from the open subsets of $\K$
and $\mathfrak{m}(\K)$ the set of all finite Borel measures on $\K$. A measure
$\phi\in\mathfrak{m}(\K)$ is \emph{positive} if $\phi(\cA)\ge 0$ for all $\cA\in
B(\K)$. Denote by $\mathfrak{m}_+(\K)$ the set of all finite positive Borel
measures on $\K$. The support $\supp{\phi}$ of a Borel measure
$\phi\in\mathfrak{m}(\K)$ is the (unique) smallest closed set $\cA\in B(\K)$ such
that $\phi(\K\setminus \cA)=0$.

Let $\phi_{ij}\in\mathfrak{m}(\K)$, $i,j= 1,\ldots, p$. The $p\times p$
matrix-valued measure $\Phi$ on $\K$ is defined as the matrix-valued function
$\Phi\colon B(\K) \to \R^{p\times p}$ with
\[\Phi(\cA)\coloneqq[\phi_{ij}(\cA)]\in \R^{p\times p}, \quad \forall \cA\in
B(\K). \] If $\phi_{ij}=\phi_{ji}$ for all $i, j =1, \ldots, p$, we call $\Phi$
a symmetric matrix-valued measure. If $\bv^{\intercal} \Phi(\cA) \bv\ge 0$ holds
for all $\cA\in B(\K)$ and for all column vectors $\bv\in\R^p$, we call $\Phi$ a
PSD matrix-valued measure. The set $\supp{\Phi}\coloneqq\bigcup_{i,j=1}^p
\supp{\phi_{ij}}$ is called the support of $\Phi$. 
%The matrix-valued integral of $\rev{z}\in\R[\x]$ with respect to the measure $\Phi$ is
%defined by 
%\[\int_{\K} \rev{z}(\x) \ud \Phi(\x)\coloneqq\left[\int_{\K}\rev{z}(\x)\ud
%\phi_{ij}(\x)\right]_{i,j=1,\ldots,p}\in\R^{p\times p}.\]
We denote by $\mathfrak{M}^p(\K)$ (resp. $\mathfrak{M}^p_+(\K)$) the set of all $p\times p$ (resp. PSD) symmetric matrix-valued measures on $\K$. 
A \emph{finitely atomic PSD} matrix-valued measure $\Phi\in\mathfrak{M}^p_+(\K)$ is a matrix-valued measure of form $\Phi=\sum_{i=1}^t W_i \delta_{\x^{(i)}}$ 
where $W_i\in\mS_+^p$, $\x^{(i)}$'s are distinct points in $\K$, and $\delta_{\x^{(i)}}$ denotes the Dirac measure centered at $\x^{(i)}$, $i=1,\ldots,t$.

For a given matrix-valued measure $\Phi=[\phi_{ij}]\in\mathfrak{M}^p_+(\K)$, we 
define a linear functional $\mL_{\Phi}  \colon \mS[\x]^{p} \to \R$ as
\begin{equation}\label{eq::Phi}
\mL_{\Phi}(F)=\int_{\K}\ \tr{F(\x)\ud \Phi(\x)}
 =\sum_{i,j}\int_{\K}\ F_{ij}(\x)\ud \phi_{ij}(\x),
\ \ \forall F(\x)\in\mS[\x]^{p}.
\end{equation}
%A linear functional $\mathscr{L} \colon \mS[\x]^{p} \to \R$ is called a 
%tracial $\K$-moment functional if there exists a matrix-valued measure 
%$\Phi\in \mathfrak{M}_+^p(\K)$ such that $\mL=\mL_{\Phi}$
%\begin{equation}\label{eq::L}
%\supp{\Phi}\subseteq \K,\ 
%\mathscr{L}(F)=\int_{\K}\ \tr{F(\x)\ud \Phi(\x)}
% =\sum_{i,j}\int_{\K}\ F_{ij}(\x)\ud \phi_{ij}(\x),
%\ \ \forall F(\x)\in\mS[\x]^{p}.
%\end{equation}
%and $\Phi$ is called a \emph{representing measure} of $\mathscr{L}$.
Let $\bS=(S_{\a})_{\a\in\N^n}$ be a multi-indexed sequence of symmetric matrices in $\mS^{p}$.
We define a linear functional $\mL_{\bS}  \colon \mS[\x]^{p} \to \R$ in the following way:
\[
\mL_{\bS}(F)\coloneqq\sum_{\a\in\supp{F}}\tr{F_{\a}S_{\a}},\quad \forall 
F(\x)=\sum_{\a\in\supp{F}} F_{\a}\x^{\a}\in\mS[\x]^{p}.
\]
We call $\mL_{\bS}$ the \emph{Riesz functional} associated to the sequence $\bS$.
The sequence $\bS$ is called a matrix-valued $\K$-moment sequence if there exists a matrix-valued measure 
$\Phi=[\phi_{ij}]\in\mathfrak{M}^p_+(\K)$ such 
that 
\begin{equation}\label{eq::S}
\supp{\Phi}\subseteq \K\quad\text{and}\quad 
S_{\a}=\int_{\K} \x^{\a}\ud \Phi(\K)\coloneqq\left[\int_{\K}\x^{\a}\ud
\phi_{ij}(\x)\right]_{i,j\in[p]},\ \forall \a\in\N^n.
\end{equation}
The measure $\Phi\in\mathfrak{M}^p_+(\K)$ satisfying \eqref{eq::S} is called a \emph{representing measure} of $\bS$.
%For a given sequence $\bS=(S_{\a})_{\a\in\N^n}\subseteq\mS^{p}$, 
%Clearly, $\bS$ is a matrix-valued $\K$-moment sequence if and only if $\mL_{\bS}$ is a tracial $\K$-moment functional.

\begin{defn}\label{def::mm}
Given a sequence $\bS=(S_{\a})_{\a\in\N^n}\subseteq\mS^{p}$,
the associated \emph{moment matrix} $M(\bS)$ is the block matrix whose block row and block column are indexed by $\N^n$ and the $(\a, \b)$-th block entry is $S_{\a+\b}$ for all $\a, \b\in\N^n$. For $G\in\mS[\x]^{q}$,
the \emph{localizing matrix} $M(G\bS)$ associated to $\bS$ and $G$ is the block matrix whose block row and block column are indexed by $\N^n$ and 
the $(\a, \b)$-th block entry is 
$\sum_{\g\in\supp{G}}S_{\a+\b+\g}\otimes G_{\g}$ for all $\a, \b\in\N^n$. For $d\in\N$, the $d$-th order moment matrix $M_d(\bS)$ (resp. localizing matrix
$M_d(G\bS)$) is the submatrix of $M(\bS)$ (resp. $M(G\bS)$) whose block row and block column are both indexed by $\N^n_d$.
\end{defn}

Let $d_k\coloneqq\lceil\deg{G_k}/2\rceil$ for $k\in[m]$ and $d_{\K}\coloneqq\max\{d_1,\ldots,d_m\}$.
As in the scalar case, the existence of a finitely atomic representing measure of a matrix-valued sequence can be guaranteed by the ``flatness condition". 
\begin{thm}[\cite{GW2023}, Theorem 5]\label{th::FEC}
Given a truncated sequence $\bS=(S_{\a})_{\a\in\N^n_{2r}}\subseteq\mS^{p}$, 
$\bS$ admits an atomic representing measure $\Phi=\sum_{i=1}^t W_i\delta_{\x^{(i)}}$
with $W_i\in\mS_+^{p}$, $\x^{(i)}\in\K$ and $\sum_{i=1}^t\rank{W_i}=\rank{M_r(\bS)}$ if and only if
$M_{r}(\bS)\succeq 0$, $M_{r-d_k}(G\bS)\succeq 0$ for $k\in[m]$, and $\rank{M_r(\bS)}=\rank{M_{r-d_{\K}}(\bS)}$.
\end{thm}

We refer the reader to \cite{GW2023} for more details on matrix-valued measures/moments.

\subsection{Polynomial Matrix Optimization and the Moment-SOS Hierarchy}

The \ac{SOS} matrix can be used to solve
PMO problems. Consider the problem of \rev{finding an $\x \in \K$ to minimize} the smallest eigenvalue of $F(\x)$:
\begin{equation}\label{eq:eig_real}
    \lambda^{\star} \coloneqq \sup_{\lambda \in \R} \lambda \quad\mathrm{s.t.} \quad
     \forall \x \in \K: \ F(\x) \succeq \lambda I_p.
\end{equation}
The \ac{SOS}-restriction of Problem \eqref{eq:eig_real} is
\begin{equation}\label{eq:eig_sos}
    \sup_{\lambda \in \R} \lambda \quad\mathrm{s.t.} \quad F(\x) - \lambda I_p \in \Sigma^p[\bG],
\end{equation}
and its $r$-th order ($r\ge r_{\min}\coloneqq\max\{\lceil\deg{F}/2\rceil,d_1,\ldots,d_m\}$) \ac{SOS}-restriction is
\begin{equation}\label{eq:eig_sos_k}
\lambda_r^* \coloneqq \sup_{\lambda \in \R} \lambda \quad\mathrm{s.t.} \quad F(\x) - \lambda I_p \in \Sigma^p[\bG]_{2r}. 
\end{equation}
The dual program of \eqref{eq:eig_sos_k} (by convex duality) is an optimization problem posed over matrix-valued sequences $\bS=(S_\a)_{\a\in \N^n_{2r}}\subseteq\mS^{p}$ \cite{GW2023}:
\begin{equation}\label{eq::mPMI:mom}
\lambda_r\coloneqq\begin{cases}
\inf\limits_{\bS}&\mL_{\bS}(F)\\
\mathrm{s.t.}&M_{r}(\bS)\succeq0,\\
&M_{r-d_k}(G_k\bS)\succeq0,\quad k\in[m],\\
&\mL_{\bS}(I_p)=1.
\end{cases}
\end{equation} 
It is clear that $\lambda_r^*\le \lambda_r\le \lambda^{\star}$
for all $r\ge r_{\min}$.
When $\Sigma^p[\bG]$ is Archimedean, Scherer and Hol's Positivestellensatz
implies that $\lim_{r\to \infty} \lambda_r^* =\lim_{r\to \infty}\lambda_r=
\lambda^{\star}$. 
% The set of bounds
% $\{\lambda^{\star}_{\sos, r}\}$ and $\{\lambda^{\star}_{\mom, r}\}$ in $r$ form monotonically increasing
% sequences of lower bounds to $\lambda^{\star}$.
The relaxations \eqref{eq:eig_sos_k} and \eqref{eq::mPMI:mom} indexed by relaxation order $r$ is called the matrix Moment-\ac{SOS} hierarchy for Problem \eqref{eq:eig_real}.

\subsection{Graph Theory}

The term-sparse, correlatively-sparse, and matrix-sparse decomposition methods for \acp{PMI} explored in this paper use graph theory to identify exploitable structure in the optimization problems.
A graph $\gs(\vs, \es)$ is defined by a collection of nodes $\vs$ and edges $\es$. This paper will consider \textit{undirected} graphs with self-loops. 
A complete graph is a graph in which every node is connected to every other node. A \textit{clique} of $\gs$ is a subgraph of $\gs$ that is isomorphic to a complete graph. A \textit{maximal clique} of $\gs$ is a clique that is not a subgraph of another clique of $\gs$.

A \textit{path} in $\gs$ exists between nodes $v_i$ and $v_j$ if there exists a sequence of nodes $v_i,v_1, v_2, \ldots, v_q,v_j$ such that $\{v_i, v_1\} \in \es, \{v_1, v_2\} \in \es, \ldots, \{v_q, v_j\} \in \es$. The nodes $v_i$ and $v_j$ are thus path-connected. The \textit{block closure} of a graph $\gs$ is the unique supergraph $\gs' \supseteq \gs$ where an edge is added between every pair of nodes that are path-connected. The graph $\gs'$ is therefore isomorphic to a direct sum of complete graphs.

A \textit{cycle} in a graph is a path that starts and ends with the same node.
A graph is \textit{chordal} if all its cycles of length at least four have a chord (i.e., an edge $\{v_i, v_j\}$ that joins two nonconsecutive nodes in the cycle). 
% Chordal graphs are also referred to as `triangulated' graphs, because all cycles can therefore be steadily decomposed into triangles. 
A \textit{chordal extension} of $\gs$ is a supergraph $\gs' \supseteq \gs$ with $\gs'$ being a chordal graph. Determining if a graph is chordal can be done in linear time, but finding a chordal extension with a minimal number of edges is NP-hard \cite{mulzer2008minimum}. Chordal extensions are typically not unique, and the block closure is a specific instance of a chordal extension. To distinguish from the block closure, we refer to an ordinary chordal extension by a \emph{chordal closure}\footnote{\rev{In this paper, we implement the ``MinimumDegree'' greedy algorithm in \cite{treewidth} to generate a chordal closure.}}. \rev{The block closure is the largest chordal extension of $\gs$ such that disconnected components of $\gs$ remain disconnected in $\gs'$.}

\subsection{Sparse Semidefinite Programming}
Chordal sparsity can be used to reduce the computational complexity needed to solve large-scale semidefinite programming problems.
Given an undirected graph $\mG(\vs,\mE)$ with nodes
$\vs=\{1,\ldots,p\}$, we denote the set of sparse symmetric matrices by $\psd^p(\mG,0)$, i.e., 
\begin{align*}
\psd^p(\mG,0)\coloneqq\{A\in\psd^p \mid A_{ij}=A_{ji}=0,\text{ if } i\neq j  \text{ and } \{i,j\}\not\in\mE\},
\end{align*}
and let $\Pi_{\mG}:\psd^p\rightarrow\psd^p(\mG,0)$ be the projection defined by 
\begin{equation}\label{sec2-eq7}
[\Pi_{\mG}(A)]_{ij}=\begin{cases}
A_{ij}, &\textrm{if }i=j\textrm{ or }\{i,j\}\in \mE,\\
0, &\textrm{otherwise}.
\end{cases}
\end{equation}

We define the cone of sparse \ac{PSD} matrices as $\psd_+^p(\mG,0)\coloneqq\psd_+^p\cap\psd^p(\mG,0)$
and the dual cone of completable \ac{PSD} matrices is given by
\begin{align*}
\psd_+^p(\mG,?)=\Pi_{\mG}(\mS_+^{p})=\{\Pi_{\mG}(A)\mid A\in\mS_+^{p}\}.
\end{align*}
For any maximal clique $\mC_i$ of $\mG(\vs,\mE)$, we define the matrix $E_{\mC_i}\in\R^{|\mC_i|\times p}$ by
\[[E_{\mC_i}]_{jk}=\left\{
		\begin{array}{ll}
		1, &\text{if }	\mC_i(j)=k,\\
		0, &\text{otherwise},
	\end{array}\right.
\]
where $\mC_i(j)$ denotes the $j$-th node in $\mC_i$, sorted in the natural ordering.
Given $A\in\psd^p$, the matrix $E_{\mC_i}$ can be used to extract a principal submatrix indexed by $\mC_i$, i.e.,
$E_{\mC_i}AE_{\mC_i}^{\intercal}\in\psd^{|\mC_i|}$. Given
$A\in\psd^{|\mC_i|}$, the operation $E_{\mC_i}^{\intercal}AE_{\mC_i}$
inflates $A$ into a sparse $p\times p$ matrix.

When the sparsity pattern graph $\mG$ is chordal, the cone $\psd_+^p(\mG,0)$ can be
decomposed as a sum of simple convex cones, as stated in the following theorem.
\begin{thm}\label{th::spositive2}{\upshape (\cite[Theorem 2.3]{AGLER1988})}
Let $\mG(\vs,\mE)$ be a chordal graph and let
$\{\mC_1,\ldots,\mC_t\}$ be the set of its maximal cliques. Then, $A\in\psd^p_+(\mG,0)$ if and only if there exist matrices $A_i\in\psd_+^{|\mC_i|}$, $i=1,\ldots,t$, such that $A=\sum_{i=1}^{t}E_{\mC_i}^{\intercal}A_iE_{\mC_i}$.
\end{thm}

By duality, Theorem \ref{th::spositive2} leads to the following characterization of the PSD completable cone $\psd_+^p(\mG,?)$.
\begin{thm}\cite[Theorem 7]{GRONE1984}\label{th::spositive}
Let $\mG(\vs,\mE)$ be a chordal graph and let $\{\mC_1,\ldots,\mC_t\}$ be the set of its maximal cliques. Then, $A\in\psd^p_+(\mG,?)$ if and only if
$E_{\mC_i}AE^{\intercal}_{\mC_i}\in\psd_+^{|\mC_i|}$ for all $i\in[t]$. 
\end{thm}

\section{Term Sparsity for Unconstrained PMO}
\label{sec:uncons}

This section will detail term sparsity methods for unconstrained PMO, or equivalently for simplifying \ac{SOS} verification of a single \ac{PMI} $F(\x) \succeq 0$ in the fashion of the decomposition in \eqref{eq:sos_pmi_uncons2}.
The first subsection will discuss how to reduce the required monomial basis, thus setting some diagonal elements of the Gram matrix $Q$ to zero. The second subsection will define matrix term sparsity patterns to impose a block structure on the matrix $Q$.

\subsection{Selecting the Monomial Basis}
\label{sec:newton}
Letting $F(\x) \in \psd^p[\x]_{2d}$ be a polynomial matrix, define $\{\bv^{j}(\x)\}_{j=1}^p$ as a set of monomial (column) vectors such that
\begin{align}\label{eq:sos_pmi_uncons_monom}
    F(\x) = \diag{\bv^{1}(\x), \bv^{2}(\x), \ldots, \bv^{p}(\x)}^{\intercal}Q \diag{\bv^{1}(\x), \bv^{2}(\x), \ldots, \bv^{p}(\x)},
\end{align}
where $\diag{\bv^{1}(\x), \bv^{2}(\x), \ldots, \bv^{p}(\x)}$ is the matrix formed by putting $\bv^{j}(\x)$ at the $j$-th diagonal position and all other elements being zeros.
The formulation in \eqref{eq:sos_pmi_uncons2} is an instance where $\forall k \in [p]: \ \bv^{j}(\x) = \bm_d(\x)$ (the full monomial basis). The \ac{PMI} $\forall  \x\in \R^n: \ F(\x) \succeq 0$ is verified if the Gram matrix satisfies $Q \succeq 0$. This subsection will use the monomial structure of $F(\x)$ to design bases $\{\bv^{j}(\x)\}$ that may have a lower cardinality than the full basis $\bm_d(\x)$. The basis reduction will be accomplished by using Newton Polytope techniques on the scalarized \ac{PMI} \eqref{eq:pmi_uncons_scalarize_sos}.

% \subsubsection{Newton Polytope Background}
% We first introduce background material on Newton Polytopes and their importance in sparse \ac{SOS} representations.

\begin{defn}
The \textit{Newton polytope} of a scalar polynomial $f \in \R[\x]$ is $\mathrm{New}(f) = \mathrm{conv}(\supp{f})$, where $\mathrm{conv}(\cdot)$ means taking the convex hull.
\end{defn}

% \begin{exmp}
%     The Newton polytope of the Motzkin polynomial $f = 1 + x_1^2 x_2^4 + x_1^4 x_2^2 - 3 x_1^2 x_2^2$ is 
%     \begin{align}
%         \mathrm{New}(f) = \mathrm{conv}\left(\left\{\begin{bmatrix}
%             0 \\ 0
%         \end{bmatrix}, \begin{bmatrix}
%             2 \\ 4
%         \end{bmatrix}, \begin{bmatrix}
%             4 \\ 2
%         \end{bmatrix} \right\}\right). \label{eq:newton_polytope_example}
%     \end{align}
%     Note that the point $[2, 2]^{\intercal}$ from the monomial $x_1^2 x_2^2$ is in the interior of the convex hull established by \eqref{eq:newton_polytope_example}.
% \end{exmp}

Newton polytopes can be used to reduce the number of active monomials in an \ac{SOS} decomposition.

\begin{thm}[Theorem 1 of \cite{Reznick_1978}]
\label{thm:newton_scalar}
Let $f$ be a polynomial such that $f \in \Sigma[\x]_{2d}$, and let $g \in \R[\x]^t$ be a vector of polynomials such that $f(\x) = \sum_{j=1}^t g_j(\x)^2$. The supports of each factor $g_j$ are constrained by
\begin{align}
    \forall j \in [t]: \qquad \mathrm{New}(g_j) \subseteq \frac{1}{2}\mathrm{New}(f). 
    \end{align}    
\end{thm}

A Gram \ac{SOS} representation $f(\x) = \bv(\x)^{\intercal} Q \bv(\x)$ can therefore restrict the monomial vector $\bv(\x)$ to all integer points of $\frac{1}{2}\mathrm{New}(f)$, rather than choosing all $\binom{n+d}{d}$ monomial elements in the full basis $\bm_d(\x)$. The size of the Gram matrix $Q$ can therefore be reduced (ensuring that solving the \ac{SDP} is more efficient), at the preprocessing cost of finding all integer points in the $1/2$-scaled Newton polytope.

% \subsubsection{Newton Polytopes for Polynomial Matrix Inequalities}

We now return to the \ac{PMI} setting with a polynomial matrix $F(\x) \in \psd^p[\x]$. Monomial bases $\{\bv^{j}(\x)\}$ from \eqref{eq:sos_pmi_uncons_monom} can be determined by virtue of Theorem \ref{thm:newton_scalar} as follows.
\begin{thm}\label{thm:newton_matrix}
Let $F(\x) \in \Sigma^p[\x]$ be a polynomial matrix and the exponent set $\cA^j$ be chosen \rev{as} $ \cA^j\coloneqq \N^{n} \cap \frac{1}{2}\mathrm{New}(F(\x)_{jj})$, where $F(\x)_{jj}$ is the $j$-th diagonal element of $F(\x)$, $j\in[p]$.
Then the monomial basis $\bv^{j}(\x)$ from \eqref{eq:sos_pmi_uncons_monom} can be chosen without conservatism as $\bv^{j}(\x) = \{\x^\a \mid \a \in \cA^j\}$.    
\end{thm}
\begin{proof}
Consider the scalarization $\y^{\intercal} F(\x) \y$ of $F(\x)$ from \eqref{eq:pmi_uncons_scalarize}. Since $\y^{\intercal} F(\x) \y$ is a quadratic form in $\y$, it holds that each integral point $(\a_\x, \a_\y) \in \N^{n+p} \cap \frac{1}{2}\mathrm{New}(\y^{\intercal} F(\x) \y)$ satisfies $\abs{\a_\y} = 1$. 
For each $j\in[p]$, let $\tilde{\cA}^j \coloneqq \{\a_\x \in \N^n \mid (\a_\x, \a_\y) \in \N^{n+p} \cap \frac{1}{2}\mathrm{New}(\y^{\intercal} F(\x) \y), \a_{y_j} = 1 \}$ corresponding to the set of $\x$-monomials multiplied by $y_j$. 
By Theorem \ref{thm:newton_scalar}, one can choose $\bv^{j}(\x) = \{\x^\a \mid \a \in \tilde{\cA}^j\}$ without conservatism. Moreover, the sets satisfy $\tilde{\cA}^j=\cA^j$ by construction.
\end{proof}

The procedure in Theorem \ref{thm:newton_matrix} will be outlined with an example.

\begin{exmp}
    Let $F(\x) \in \psd^3[\x]$ be defined by
    \begin{equation}
        F(\x) = \begin{bmatrix}
            1+x_1^4 - 0.5 x_1^2 + 0.25x_1 + x_2^2 + x_3^4 & -x_2 & x_1+x_3 \\
            -x_2 & 2+2x_2^4 + 2 x_1^2 - x_1 & x_1 \\
            x_1+x_3& x_1 & 2 + 3 x_3^2 + x_1^2 x_2^2
        \end{bmatrix}. \label{eq:f3_exmp}
    \end{equation}
By respectively computing the Newton polytopes of the diagonal elements $F(\x)_{11},$ $F(\x)_{22},$  $F(\x)_{33}$, 
the monomial bases $\{\bv^{j}(\x)\}$ can be given as
\begin{subequations}
\label{eq:newton_group}
    \begin{align}
    \bv^{1}(\x) &= \{1, x_1^2, x_1, x_1 x_3, x_2, x_3^2, x_3\}, \\
    \bv^{2}(\x) &= \{1, x_1, x_2, x_2^2\}, \\
    \bv^{3}(\x) &= \{1, x_1 x_2, x_3\}.
    \end{align}
\end{subequations}
There are 14 monomials in total listed in \eqref{eq:newton_group}. Verification that $F(\x)$ from \eqref{eq:f3_exmp} is an \ac{SOS} matrix can therefore be accomplished by solving an \ac{SDP} with a Gram matrix $Q \in \psd_+^{14}$ rather than in $ \psd_+^{30}$ (as would be required with a full choice of all $\binom{3+2}{3}=10$ monomials for each of the three columns). An admissible \ac{SOS} decomposition $F(\x)= R(\x)^{\intercal} R(\x)$ found by the reduced bases in \eqref{eq:newton_group} is 
\begin{align}
    R(\x) \approx \begin{bmatrix}
        1 + 0.125 x_1 - 0.517x_1^2 & 0.0864 x_1& x_3\\
        0.0760x_1 + 0.8543 x_1^2 & -0.617 + 0.0525 x_1& 0.608 x_3\\
        0.707 x_1 & 0.707 x_1 & 1.414\\ 
        x_2 & -1 & 0\\
        x_3^2  & 0 & 0\\
        -0.050 x_1 & 0.787 - 0.594 x_1 & 0.477 x_3 \\
        -0.016 x_1 & 1.066 x_1 & 0.155 x_3\\ 
        0 & 1.414 x_2^2 & 0\\
        -0.1223 x_1 & 0 & 1.174 x_3\\
        0 & 0 & x_1 x_2
    \end{bmatrix}.
\end{align}
    
\end{exmp}

\subsection{Term Sparsity Patterns}\label{sec:block_diag}
The previous subsection concentrated on reducing the number of elements in the monomial bases $\{\bv^{j}(\x)\}$. This subsection will focus on reducing the maximal size of \ac{PSD} matrices involved in the \ac{SDP} through the application of term sparsity.

% \begin{defn}
%     A \textbf{sign-symmetry} of a polynomial matrix $F \in \psd^p[\x]$ is a binary vector $\theta \in \{-1, 1\}^n$ such that 
%     \begin{equation}
%         \forall \x \in \R^n: \qquad \lambda_{\min}(F(\x)) = \lambda_{\min}(F(\theta \odot \x))
%     \end{equation}
% \end{defn}

% Particular instances of sign-symmetries include even symmetries ($F(\x) = F(-\x)$) or coordinate-wise symmetries (e.g. $F(x_1, x_2, x_3) = F(-x_1, x_2, x_3)$). As the following example demonstrates, it is not necessary that $F(\x) = F(\theta \odot \x):$
% \begin{exmp}
%     Consider the univariate \ac{PMI} defined by
%     \begin{equation}
%         F(x) = \begin{bmatrix}
%             1 & x \\ x & 1
%         \end{bmatrix} \succeq 0. \label{eq:pmi_demo}
%     \end{equation}

%     The validity region of this \ac{PMI} is $x \in [-1, 1]$ (which can also be expressed through a Schur complement as $1-x^2 \geq 0$). The \ac{PMI} in \eqref{eq:pmi_demo} obeys a sign symmetry with
%         \begin{equation}
%         F(-\x) = \begin{bmatrix}
%             1 & -x \\ -x & 1
%         \end{bmatrix} \succeq 0. \label{eq:pmi_demo_flip}
%     \end{equation}
%     While $F(x) \neq F(-x)$ except when $x=0$, it does hold that $\forall x \in \R: \ \lambda_{\min}(F(x)) = \lambda_{\min}(F(-x))$.
% \end{exmp}

% A matrix term sparsity pattern is a graph $\gs(\vs, \es)$ with adjacency matrix $\mathcal{C}$ in 

% which could be the set of all monomials up to degree $r$, or could arise from a Newton-Polytope-based preprocessing method.

\begin{defn}
\rev{Let $\bB \coloneqq \{\bv^1(\x),\ldots,\bv^{p}(\x)\}$ is the concatenation of the vectors $\bv^i(\x)$.}
The \emph{term sparsity pattern (TSP)} of a polynomial matrix $F(\x) \in \psd^p[\x]$ is an undirected graph including self-loops with nodes corresponding to entries of $\bB$, and edges are drawn between nodes $a\in\bv^{i}(\x)$ and $b\in\bv^{j}(\x)$ if $ab \in \{\x^\a\mid\a\in\supp{F_{ij}}\}$ or if $ab$ is a monomial square given that $i=j$.
\end{defn}

\begin{exmp}\label{exmp:bivariate}
Consider the bivariate matrix 
\begin{equation}
F(\x) = \begin{bmatrix}
            1 + x_1^2 + 2 x_1^2 x_2^2 + x_2^2 & x_1 x_2 \\
            x_1 x_2 & 2 + x_1^2 x_2^2 + x_2^4
\end{bmatrix} \label{eq:bivariate_example}
\end{equation}
 admitting an \ac{SOS}-matrix decomposition $F(\x)=R(\x)^{\intercal}R(\x)$ with
    \begin{align}
        R(\x) &= \begin{bmatrix}
            x_1 & x_2 \\
            \sqrt{2} x_1 x_2 & 0 \\
            x_2 &  0\\
            1 & 0\\
            0 & \sqrt{3/2}\\
            0 & x_1 x_2 \\
            0 & x_2^2/\sqrt{2} \\
            0 & (x_2^2-1)/\sqrt{2}
        \end{bmatrix}.
    \end{align}

This example uses the standard full bases $\bv^{1}(\x) = \bv^{2}(\x)=[1, x_1, x_2, x_1^2, x_1 x_2, x_2^2]$.
The $j$-indices of a monomial $a\in\bv^{j}(\x)$ for $j \in \{1, 2\}$ will be indicated by a bracketed-superscript as $[a]^j$.
The TSP of \eqref{eq:bivariate_example} is drawn in Figure \ref{fig:TSP_bivariate_first_chordal}. This graph has 5 connected components. As an example, the off-diagonal term $F(\x)_{12} = x_1 x_2$ could be created by multiplications of $[1]^1,[x_1x_2]^2$ or $[x_1]^1, [x_2]^2$.

\end{exmp}

% Given an adjacency matrix $\mathcal{C}$ for a graph in which the nodes are labeled by monomials, we can define the `support' of the graph as the set of block-indexed set of monomials:
% \begin{align}
%     \supp{\mathcal{\gs}} = \{[\x^\alpha]_{jk} \mid \{[\x^\a]^i, [\x^\b]^j\} \in \es\}    \label{eq:supp_from_graph}
% \end{align}

The term sparsity method proceeds by the iterative repetition of two operations on the TSP graph \cite{wang2021tssos}:
\begin{itemize}
    \item \textbf{Support Extension:} Add edges to the graph based on the sets of active monomials;
    \item \textbf{Chordal Extension:} Add edges to the graph by performing a chordal extension.
\end{itemize}

More concretely, for each pair $i,j\in[p]$, define the initial support set associated to the $(i,j)$-position by
\begin{equation}
    \cC_{i,j}^{(0)}\coloneqq \begin{cases}\supp{F_{ii}}\cup\bv^{i}(\x)^{\circ2}, \quad&\text{if }i=j,\\
    \supp{F_{ij}},\quad&\text{otherwise},
    \end{cases}
\end{equation}
where $\bv^{i}(\x)^{\circ2}$ denotes the set of squares of monomials in $\bv^{i}(\x)$.
For $s\ge0$, the support extension operation is defined as
\begin{equation}
\text{SupportExtension}(\es^{(s)})\coloneqq\bigcup_{i,j=1}^p\left\{\left\{[a]^i,[b]^j\right\}\in\bv^i(\x)\times\bv^j(\x)\,\middle|\, ab \in \cC_{i,j}^{(s)}\right\}.
\end{equation}
Next, we perform a chordal extension on the graph with $\text{SupportExtension}(\es^{(s)})$ as its edges:
\begin{equation}\label{sec3:eq1}
    \es^{(s+1)} \leftarrow \text{ChordalExtension}(\text{SupportExtension}(\es^{(s)})), 
\end{equation}
and update the support set 
\begin{equation}
    \cC_{i,j}^{(s+1)}\coloneqq \left\{ab\,\middle|\,\left\{[a]^{i},[b]^{j}\right\}\in \es^{(s+1)}\right\}.
\end{equation}
% The term sparsity process \cite{wang2021tssos} consists of the following fixed-point iteration (with respect to Block closure or Chordal closure):
% The fixed-point process can also terminate prematurely before convergence (such as in $s=1$ or $s=2$ iterations).

There are two typical choices for the chordal extension operation in \eqref{sec3:eq1}: an (approximately) smallest chordal closure (i.e., chordal extension with the smallest tree-width) and the block closure. The block closure would offer a block-diagonal structure on the Gram matrix, whereas the sparser chordal closure could potentially contain overlaps among maximal cliques. The chordal extension does not add any edges to the TSP graph in Figure \ref{fig:PMI-sparsity}. 
Figure \ref{fig:TSP_bivariate_first_closure} \rev{displays the block closure, and highlights new edges as red lines.} 
% \ref{exmp:TSP_bivariate_first_block}.
\begin{figure}[!htbp]
    \begin{subfigure}[b]{0.45\textwidth}
    \begin{center}
{
{\footnotesize\begin{tikzpicture}[every node/.style={circle, draw=blue!50, thick}]
\node (n1) at (2,2) {$[1]^1$};
\node (n4) at (2,4) {$[x_1^2]^1$};
\node (n6) at (4,2) {$[x_2^2]^1$};
\node (n11) at (0,2) {$[x_1 x_2]^2$};
\node (n9) at (6,0) {$[x_2]^2$};
\node (n2) at (6,-2) {$[x_1]^1$};
\node (n7) at (2,0) {$[1]^2$};
\node (n5) at (0,0) {$[x_1 x_2]^1$};
\node (n8) at (6,4) {$[x_1]^2$};
\node (n3) at (6,2) {$[x_2]^1$};
\node (n10) at (4,0) {$[x_1^2]^2$};
\node (n12) at (2,-2) {$[x_2^2]^2$};
\draw (n1)--(n4);
\draw (n1)--(n6);
\draw (n4)--(n6);
\draw (n1)--(n11);
\draw (n9)--(n2);
\draw (n7)--(n5);
\draw (n8)--(n3);
\draw (n10)--(n12);
\draw (n10)--(n7);
\draw (n12)--(n7);
\end{tikzpicture}}}\caption{chordal closure\label{fig:TSP_bivariate_first_chordal} }
\end{center}
    \end{subfigure}
    \hfill
    \begin{subfigure}[b]{0.45\textwidth}
\begin{center}
{{\footnotesize
\begin{tikzpicture}[every node/.style={circle, draw=blue!50, thick}]
\node (n1) at (2,2) {$[1]^1$};
\node (n4) at (2,4) {$[x_1^2]^1$};
\node (n6) at (4,2) {$[x_2^2]^1$};
\node (n11) at (4,4) {$[x_1 x_2]^2$};
\node (n9) at (6,0) {$[x_2]^2$};
\node (n2) at (6,-2) {$[x_1]^1$};
\node (n7) at (2,0) {$[1]^2$};
\node (n5) at (4,-2) {$[x_1 x_2]^1$};
\node (n8) at (6,4) {$[x_1]^2$};
\node (n3) at (6,2) {$[x_2]^1$};
\node (n10) at (4,0) {$[x_1^2]^2$};
\node (n12) at (2,-2) {$[x_2^2]^2$};
\draw [color=red](n1)--(n4);
\draw [color=red](n1)--(n6);
\draw (n4)--(n6);
\draw (n1)--(n11);
\draw (n6)--(n11);
\draw (n4)--(n11);
\draw (n9)--(n2);
\draw (n7)--(n5);
\draw (n8)--(n3);
\draw (n10)--(n12);
\draw (n10)--(n7);
\draw (n12)--(n7);
\draw [color=red](n12)--(n5);
\draw [color=red](n10)--(n5);
\end{tikzpicture}}}\caption{block closure\label{fig:TSP_bivariate_first_block} }
\end{center}
\end{subfigure}
\caption{TSP graph of $F(\x)$ in \eqref{eq:bivariate_example} and its chordal/block closure}
\label{fig:TSP_bivariate_first_closure}
\end{figure}

The process in \eqref{sec3:eq1} is guaranteed to stabilize after a finite number of iterations, because the cardinality of the node set $\bB$ is finite and both the support extension and chordal extension steps only add edges. After obtaining the graph $\G^{(s)}(\bB,\es^{(s)})$ at the $s$-th term sparsity step ($s\ge1$), we may consider the following sparse SDP to decompose $F(\x)$ as an SOS:
\begin{equation}\label{eq::sparseSOS}
\begin{cases}
   \text{Find}&Q\in\mathbb{S}^{|\bB|}_+(\G^{(s)},0)\\
    \text{s.t.}&F(\x) = \diag{\bv^{1}(\x), \bv^{2}(\x), \ldots, \bv^{p}(\x)}^{\intercal}Q \diag{\bv^{1}(\x), \bv^{2}(\x), \ldots, \bv^{p}(\x)}.
\end{cases}
\end{equation}
\begin{thm}
    Assume that the chordal extension in \eqref{sec3:eq1} is chosen to be the block closure
    and the process in \eqref{sec3:eq1} stabilizes at the $s$-th step. Then, 
    $F(\x) \in \Sigma^p[\x]$ if and only if \eqref{eq::sparseSOS} is feasible.
\end{thm}
\begin{proof}
    The proof is similar to the counterpart in \cite[Theorem 3.3]{wang2021tssos}.
\end{proof}

% The sparse moment relaxation at the $s$-th term sparsity iteration for an unconstrained polynomial matrix 
% optimization $\inf_{\x\in\R^n}\lambda_{\min}(F(\x))$ 
% is the optimization problem posed over pseudomoment-sequences $\bS=\{S_\a\}_{\a\in \N^n_{2r}}\subseteq\mS^{p}$
% \begin{equation}\label{sec3:eq2}
% \begin{cases}
% \inf_{\bS}&\mL_{\bS}(F)\\
% \mathrm{s.t.}&B_{\G^{(s)}}\circ M_{r}(\bS)\in\mathbb{S}^{|\bB|}_+(\G^{(s)},?),\\
% &\mL_{\bS}(I_p)=1.
% \end{cases}
% \end{equation}
% in which information about the cone of completable matrices is available in Appendix \ref{app:sparse}. 

\section{Term Sparsity for Constrained PMO}
\label{sec:cons}

This section will apply term sparsity techniques towards PMO constrained by \acp{PMI}. Specifically, we extend the iterative procedure on exploiting term sparsity for scalar constrained polynomial optimization \cite{wang2021tssos,wang3} to the situation of PMO. Just as in the unconstrained case of Section \ref{sec:block_diag}, the term sparsity decomposition will proceed based on alternating support extension and chordal extension steps.
% The section will open with a brief discussion about the difficulty of applying Newton polytope methods in the constrained case.
% The focus of this section will be placed on sign-symmetry based block diagonalization methods. 

Consider the PMO problem
\begin{equation}\label{eq::mPMI}
	\lambda^{\star}\coloneqq\inf_{\x\in\R^n}\lambda_{\min}(F(\x))\quad \text{s.t. }\  G_1(\x)\succeq 0,\ldots,G_m(\x)\succeq 0, 
\end{equation}
where $F=[F_{ij}]\in\mS^p[\x]$ and $G_k=[G_k^{i,j}]\in\mS^{q_k}[\x], k\in[m]$.
By abuse of notation, we denote a monomial basis $\{\x^{\a}\}_{\a\in\cB}$ by its exponent set $\cB$ in the following. For a monomial basis $\cB$ and a positive integer $k$, we use $\cB^{(k)}$ to denote the vector of $k$-copies of $\cB$, that is,
\begin{equation}
\cB^{(k)}\coloneqq(\underbrace{\a,\ldots,\a}_{k\text{ copies}})_{\a\in\cB}.
\end{equation}
For convenience, the $k$-copies of $\a\in\cB$ in $\cB^{(k)}$ are respectively labeled $[\a]^1,\ldots,[\a]^k$.

Recall that $d_k\coloneqq\lceil\deg{G_k}/2\rceil$ for $k\in[m]$ and $r_{\min}\coloneqq\max\{\lceil\deg{F}/2\rceil,d_1,\ldots,d_m\}$. 
Fix a relaxation order $r\ge r_{\min}$. Set $d_0\coloneqq0$ and $q_0\coloneqq1$. Let
\begin{equation}
\bB_{r,k}\coloneqq (\N^n_{r-d_k})^{(pq_k)}, \quad k=0,\ldots,m,  
\end{equation}
and for each pair $i,j\in[p]$, let us define the initial support set
\begin{equation}
    \cC_{i,j}^{(0)}\coloneqq \begin{cases}\supp{F_{ii}}\cup(2\N^n_{r}), \quad&\text{if }i=j,\\
    \supp{F_{ij}},\quad&\text{otherwise}.
    \end{cases}
\end{equation}
% Define the {\em term sparsity pattern (tsp) graph}, denoted by $\G_{r}^{\mathrm{tsp}}$, such that the node set $V(\G_{r}^{\mathrm{tsp}})=\bB_{r,0}$ and the edge set
% \begin{equation}
% E(\G_{r}^{\mathrm{tsp}})=\left\{\{a,b\}\middle|\text{ there exist $i,j\in[p]$ such that }a\in \bm^i_{r}(\x),b\in \bm^j_{r}(\x)\text{ and }a\cdot b\in\A_{ij}\right\}.
% \end{equation} 

Now for each $k\in[m]$, we iteratively define an $s$-indexed  ascending chain of graphs $\left(\G_{r,k}^{(s)}\left(\bB_{r,k},\es_{r,k}^{(s)}\right)\right)_{s\ge1}$ via two successive operations:\\
1) {\bf Support Extension.} Add edges to the graph $\G_{r,k}^{(s)}$ according to the sets of activated supports $\cC_{i,j}^{(s)}$:
\begin{equation*}
\mathrm{SupportExtension}(\es_{r,k}^{(s)})\coloneqq
\bigcup_{i,j=1}^{pq_k}\bigg\{\left\{[\a]^i,[\b]^j\right\}\in\N^n_{r-d_k}\times\N^n_{r-d_k}\bigg|\left(\a+ \b+\supp{G_k^{\bar{i},\bar{j}}}\right)\cap\cC_{\lceil i/q_k\rceil,\lceil j/q_k\rceil}^{(s)}\ne\emptyset\bigg\},
\end{equation*}
where for any $i\in[pq_k]$,
\begin{equation*}
    \bar{i}\coloneqq\begin{cases}i\,(\text{mod }q_k),\quad&\text{if }q_k\nmid i,\\
    q_k,\quad&\text{otherwise}.
    \end{cases}
\end{equation*}
2) {\bf Chordal Extension.} For each $k$, perform a chordal extension on the graph with $\mathrm{SupportExtension}(\es_{r,k}^{(s)})$ as its edge set, i.e., let
\begin{equation}\label{sec4:eq1}
\es_{r,k}^{(s+1)}\leftarrow \text{ChordalExtension}(\text{SupportExtension}(\es_{r,k}^{(s)})),\quad k=0,\ldots,m,
\end{equation}
and for each $i,j\in[p]$, update the sets of activated supports:
\begin{equation}
    \cC_{i,j}^{(s+1)}\coloneqq \bigcup_{k=0}^m\bigcup_{i',j'=1}^{q_k}\bigcup_{\{[\a]^{(i-1)q_k+i'},[\b]^{(j-1)q_k+j'}\}\in \es_{r,k}^{(s+1)}}\left(\a+ \b+\supp{G_k^{i',j'}}\right).
\end{equation}
Clearly, the process in \eqref{sec4:eq1} is guaranteed to stabilize after a finite number of iterations for all $k$, because the cardinality of each node set $\bB_{r,k}$ is finite and both the support extension and chordal extension steps only add edges.

Given an undirected graph $\gs(\vs, \es)$, we use the symbol $B_{\gs} \in \{0, 1\}^{\abs{\vs}\times\abs{\vs}}$ to refer to the adjacency matrix of $\gs$ whose diagonal are all ones. Then, for $s\ge1$, the corresponding sparse moment relaxation for \eqref{eq::mPMI} is given by
\begin{equation}\label{sec4:eq2}
\lambda^{(s)}_{r}\coloneqq\begin{cases}
\inf&\mL_{\bS}(F)\\
\mathrm{s.t.}&B_{\G_{r,0}^{(s)}}\circ M_{r}(\bS)\in\mathbb{S}_+(\G_{r,0}^{(s)},?),\\
&B_{\G_{r,k}^{(s)}}\circ M_{r-d_k}(G_k\bS)\in\mathbb{S}_+(\G_{r,k}^{(s)},?),\quad k\in[m],\\
&\mL_{\bS}(I_p)=1.
\end{cases}
\end{equation}
We call $s$ the \emph{sparse order} associated with Problem \eqref{sec4:eq2}.

% For each $k$, one can write $M_{r-d_k}(G_k\bS)=\sum_{i,j,\a}D_{ij\a}^{k}S^{ij}_{\a}$ for appropriate matrices $\{D_{ij\a}^{k}\}$. Then the dual of \eqref{sec4:eq2} reads as
% \begin{equation}
% \begin{cases}
% \sup\,&\lambda\\
% \textrm{s.t.}\, &\sum_{k=0}^m\langle Q_{k},D_{ij\a}^{k}\rangle=F_{i,j}^{\a},\quad\forall \a\in\cC_{i,j}^{(s)},i<j\in[p],\\
% &\sum_{k=0}^m\langle Q_{k},D_{ii\a}^{k}\rangle+\lambda\delta_{\mathbf{0}\a}=F_{i,i}^{\a},\quad\forall \a\in\cC_{i,i}^{(s)},i\in[p],\\
% &Q_{k}\in\mathbb{S}_+(\G_{r,k}^{(s)},0),\quad k=0,\ldots,m.
% \end{cases}
% \end{equation}

\begin{prop}
The following statements are true:
\begin{enumerate}
    \item[\upshape (i)] Fixing a relaxation order $r\ge r_{\min}$, the sequence $\left\{\lambda^{(s)}_{r}\right\}_{s\ge1}$ is monotonically non-decreasing and $\lambda^{(s)}_{r}\le\lambda^{\star}$ for all $s\ge1$.
    \item[\upshape (ii)] Fixing a sparse order $s\ge 1$, the sequence $\left\{\lambda^{(s)}_{r}\right\}_{r\ge r_{\min}}$ is monotonically non-decreasing.
    \item[\upshape (iii)] Assume that the chordal extension in \eqref{sec4:eq1} is chosen to be the block closure. Then for any $r\ge r_{\min}$, the sequence $\left\{\lambda^{(s)}_{r}\right\}_{s\ge1}$ converges to $\lambda_r$ in finitely many steps.
\end{enumerate}
\end{prop}
\begin{proof}
The proofs are similar to the counterparts in \cite{wang2021tssos,wang3}.
\end{proof}

\rev{
Term sparse decompositions are intricately linked to \emph{PMI sign symmetries} of the underlying polynomial matrices.

% \begin{defn}\label{def::ss}
% A \emph{sign symmetry} of a polynomial matrix $P(\x) \in \psd^p[\x]$ is a binary vector $\theta \in \{-1, 1\}^n$ such that 
%     \begin{equation}
%         P(\x) = P(\theta \circ \x).
%     \end{equation}
% The sign symmetries of $P(\x)$ is the set of all such binary vectors.
% \end{defn}

\begin{defn}\label{def::ss1}
A \emph{PMI sign symmetry} of a polynomial matrix $P(\x) \in \psd^p[\x]$ is a binary vector $\theta \in \{-1, 1\}^n$ such that either $P(\theta \circ \x)=P(\x)$ or there exists a complete bipartite graph $\gs^{\theta}(\vs, \es)$ with $\vs=[p]$ and satisfying
\begin{enumerate}
    \item[\rm (1)] $[P(\theta \circ \x)]_{ij} = [P(\x)]_{ij}$ if $i=j$ or $\{i,j\}\notin\es$;
    \item[\rm (2)] $[P(\theta \circ \x)]_{ij} = -[P(\x)]_{ij}$ if $\{i,j\}\in\es$.
\end{enumerate}
In addition, we define the binary matrix $E^{P,\theta}\in \{-1, 1\}^{p\times p}$ such that $P(\theta \circ \x)=E^{P,\theta}\circ P(\x)$.
% The PSD sign symmetries of $P(\x)$ is the set of all such binary vectors.
\end{defn}

The name ``PMI sign symmetry'' is justified by the following proposition.
\begin{prop}\label{prop:ss}
Let $P(\x) \in \psd^p[\x]$ and $\theta$ be a PMI sign symmetry of $P(\x)$. If $P(\x)\succeq0$, then $P(\theta \circ \x)\succeq0$.
\end{prop}
\begin{proof}
Assume that $E^{P,\theta}$ is not an all-one matrix. One can easily show that $E^{P,\theta}\succeq0$ from the fact that $\gs^{\theta}(\vs, \es)$ is a complete bipartite graph. Therefore, we have $P(\theta \circ \x)=E^{P,\theta}\circ P(\x)\succeq0$.
\end{proof}

The PMI sign symmetry is a generalization of the usual sign symmetry (corresponding to $P(\theta \circ \x)=P(\x)$). The latter was studied in \cite{lofberg2008block} to yield a block-diagonal structure for matrix SOS decompositions. It turns out that the more general PMI sign symmetry can also give rise to a block-diagonal structure for matrix SOS decompositions.

\begin{thm}\label{sec4:thm0}
Suppose that $F(\x) \in \psd^p[\x]$ admits a representation:
\begin{equation}
    F(\x) = \diag{\bv^{1}(\x), \bv^{2}(\x), \ldots, \bv^{p}(\x)}^{\intercal}Q\diag{\bv^{1}(\x), \bv^{2}(\x), \ldots, \bv^{p}(\x)}
\end{equation}
with $Q\succeq0$. Let $\Theta$ collect all PMI sign symmetries of $F(\x)$. We define a $\{0,1\}$-binary matrix $B$ indexed by $\{\bv^{1}(\x), \bv^{2}(\x), \ldots, \bv^{p}(\x)\}$ via
\begin{equation}
    B_{[a]^i,[b]^j}=\begin{cases}
1,\quad&\text{ if }\left([a]^i\cdot[b]^j\right)(\theta \circ \x)=E^{F,\theta}_{ij}\cdot\left([a]^i\cdot[b]^j\right)(\x),\quad\forall\theta\in\Theta,\\
0,\quad&\text{ otherwise.}\end{cases}
\end{equation}
Then, it holds
\begin{equation}\label{ss1}
    F(\x) = \diag{\bv^{1}(\x), \bv^{2}(\x), \ldots, \bv^{p}(\x)}^{\intercal}\left(B\circ Q\right)\diag{\bv^{1}(\x), \bv^{2}(\x), \ldots, \bv^{p}(\x)}.
\end{equation}
\end{thm}
\begin{proof}
Let us consider the scalarization $\y^\intercal F(\x)\y$ of $F(\x)$. It can be seen that the PMI sign symmetries of $F(\x)$ are in a one-to-one correspondence to the usual sign symmetries of $\y^\intercal F(\x)\y$. Then, the block-diagonal structure in \eqref{ss1} follows from the block-diagonal structure given by the sign symmetries of $\y^\intercal F(\x)\y$ \cite{lofberg2009pre}.
\end{proof}

Theorem \ref{prop:ss} can be further extended to the constrained case as follows.
\begin{thm}\label{sec4:thm1}
Let $F\in\mS^p[\x]$ and $G_k\in\mS^{q_k}[\x], k\in[m]$. Suppose that $F$ admits a representation:
\begin{equation}
F=\left(I_p\otimes \bm_r(\x) \right)^{\intercal} Q_0 \left(I_p\otimes \bm_r(\x)\right)+\sum_{k=1}^m\left\langle \left(I_{pq_k}\otimes \bm_{r-d_k}(\x) \right)^{\intercal} Q_k \left(I_{pq_k}\otimes \bm_{r-d_k}(\x)\right), G_k\right\rangle_p,
\end{equation}
where $Q_0,Q_1,\ldots,Q_m\succeq0$. Let $\Theta$ collect the common PMI sign symmetries of $F,G_1,\ldots,G_m$. Let $B_0$ be defined in the same manner as $B$ in Theorem \ref{sec4:thm0}. Moreover, for each $k\in[m]$, we define a $\{0,1\}$-binary matrix $B_k$ indexed by $\{\bm^{1}_{r-d_k}(\x), \ldots, \bm_{r-d_k}^{pq_k}(\x)\}$ via
\begin{equation}
    [B_k]_{[a]^{(i-1)q_k+\ell},[b]^{(j-1)q_k+w}}=\begin{cases}
1,\quad&\text{ if }c(\theta \circ \x)=E^{F,\theta}_{ij}\cdot E^{G_k,\theta}_{\ell w}\cdot c(\x),\quad\forall\theta\in\Theta,\\
0,\quad&\text{ otherwise,}\end{cases}
\end{equation}
where $c\coloneqq[a]^{(i-1)q_k+\ell}\cdot[b]^{(j-1)q_k+w}$.
Then, $F$ admits the following block-diagonal representation:
\begin{equation}
F=\left(I_p\otimes \bm_r(\x) \right)^{\intercal} (B_0\circ Q_0) \left(I_p\otimes \bm_r(\x)\right)+\sum_{k=1}^m\left\langle \left(I_{pq_k}\otimes \bm_{r-d_k}(\x) \right)^{\intercal} (B_k\circ Q_k) \left(I_{pq_k}\otimes \bm_{r-d_k}(\x)\right), G_k\right\rangle_p.
\end{equation}
\end{thm}

The proof of Theorem \ref{sec4:thm1} also utilizes the scalarization trick, but is more technical and we omit here for conciseness.
As a corollary of Theorem \ref{sec4:thm1}, we immediately obtain the following PMI sign symmetry adapted Positivstellensatz for polynomial matrices.
\begin{cor}
Let $F\in\mS^p[\x]$ and $G_k\in\mS^{q_k}[\x], k\in[m]$. Suppose that $F$ is positive definite on $\mathbf{K}$ defined in \eqref{eq:bsa_set} and the related quadratic module $\Sigma^p[\bG]$ satisfies the Archimedean property. Let $\{B_k\}_{k=0}^m$ be the $\{0,1\}$-binary matrices provided in Theorem \ref{sec4:thm1} by considering the common PMI sign symmetries of $F,G_1,\ldots,G_m$.
Then, $F$ admits a block-diagonal representation:
\begin{equation}
F=\left(I_p\otimes \bm_r(\x) \right)^{\intercal} (B_0\circ Q_0) \left(I_p\otimes \bm_r(\x)\right)+\sum_{k=1}^m\left\langle \left(I_{pq_k}\otimes \bm_{r-d_k}(\x) \right)^{\intercal} (B_k\circ Q_k) \left(I_{pq_k}\otimes \bm_{r-d_k}(\x)\right), G_k\right\rangle_p.
\end{equation}
for some $Q_0,Q_1,\ldots,Q_m\succeq0$.
\end{cor}
\begin{proof}
This follows from Scherer and Hol's Positivestellensatz and Theorem \ref{sec4:thm1}.
\end{proof}

% Particular instances of sign symmetries include even symmetries (i.e., $P(\x) = P(-\x)$) or coordinate-wise symmetries (e.g., $P(x_1, x_2, x_3) = P(-x_1, x_2, x_3)$). Let $R$ be the binary matrix whose columns consist of the common sign symmetries of $F(\x),G_1(\x),\ldots,G_m(\x)$.  
% We define an equivalence relation $\sim$ on $\bB_{r,k}$ ($k=0,\ldots,m$) by
% \begin{equation}
%     \a\sim \b \iff R^{\intercal}(\a+\b)\equiv0\, (\mathrm{mod}\,2).
% \end{equation}
% For each $k\in\{0\}\cup[m]$, the equivalence relation $\sim$ gives rise to a partition of $\bB_{r,k}$, and thus defines a block-diagonal structure on the moment matrix or localizing matrices (each block being indexed by an equivalence class in $\bB_{r,k}$).

% \begin{cor}
% In \eqref{eq::mPMI:mom}, there is no loss of generality in assuming that the moment matrix and localizing matrices possess the block-diagonal structure provided by the common PSD sign symmetries of $F(\x),G_1(\x),\ldots,G_m(\x)$.
% \end{cor}

% It is easily seen from the construction that the block structure at any term sparsity iteration is a refinement of the one determined by sign symmetries.
For scalar polynomial optimization, the block structures of the term sparsity iterations with block closures converge to the one determined by sign symmetries \cite{wang2021tssos}. This result could be extended to the case of PMO by considering PMI sign symmetries.

\begin{thm}
Consider the PMO problem \eqref{eq::mPMI}. The block structures of the term sparsity iterations with block closures converge to the one determined by the common PMI sign symmetries of $F(\x),G_1(\x),\ldots,G_m(\x)$.
\end{thm}
\begin{proof}
We assume that $q_1=\cdots=q_m=1$. The proof for the general case is more involved but follows in a similar manner. Now the PMO problem \eqref{eq::mPMI} is equivalent to the following scalar polynomial optimization problem:
\begin{equation}\label{eq::scalarized}
	\inf_{\x\in\R^n, \y\in\R^p} \y^\intercal F(\x)\y \quad \mathrm{s.t. }\  
    G_1(\x)\ge 0,\ldots,G_m(\x)\ge 0,\ 1-y_1^2-\cdots-y_p^2=0.
\end{equation}
Comparing the term sparsity iterations for \eqref{eq::mPMI} described in this work with the ones for \eqref{eq::scalarized} developed in \cite{wang2021tssos}, we find that two nodes $[\a]^i$ and $[\b]^j$ are connected in the matrix setting if and only if $\x^{\a}y_i$ and $\x^{\b}y_j$ are connected in the scalar setting. 
Moreover, the PMI sign symmetries of \eqref{eq::mPMI} are in a one-to-one correspondence to the usual sign symmetries of \eqref{eq::scalarized}.
Therefore,  by \cite[Corollary 6.8]{wang2021tssos},
the block structures of the term sparsity iterations with block closures for
\eqref{eq::mPMI} converge to the one 
determined by the common PMI sign symmetries of $F(\x),G_1(\x),\ldots,G_m(\x)$.
\end{proof}}

% \begin{rem}
% Consider the special PMO problem
% \begin{equation}\label{eq::notscalarized}
% 	\inf_{\x\in\R^n}\lambda_{\min}(F(\x)\quad \text{s.t. }\  g_1(\x)\ge 0,\ldots,g_m(\x)\ge 0, 
% \end{equation}
% where $F\in\mS^p[\x]$ and $g_k\in\R[\x], k\in[m]$. 
% This can be equivalently reformulated into a scalarized problem:
% \begin{equation}\label{eq::scalarized}
% 	\inf_{\x\in\R^n, \y\in\R^p} \y^\intercal F(\x))\y \quad \text{s.t. }\  
%     g_1(\x)\ge 0,\ldots,g_m(\x)\ge 0,\ 1-y_1^2-\cdots-y_p^2=0.
% \end{equation}
% We compare the term sparsity procedure in the matrix setting described in this work 
% for \eqref{eq::notscalarized} with the one in the scalar setting developed in \cite{wang2021tssos}
% for \eqref{eq::scalarized}. 
% One can verify that two nodes $[\a]^i$ and $[\b]^j$ are connected in the stabilized graph for 
% \eqref{eq::notscalarized} if and only if $[\a]^iy_i$ and $[\b]^jy_j$ are connected in that
% for \eqref{eq::scalarized}. Therefore,  by \cite[Corollary 6.8]{wang2021tssos},
% the block structures of the term sparsity iterations with block closures for the problem 
% \eqref{eq::notscalarized} converge to the one 
% determined by sign symmetries of the scalarized problem \eqref{eq::scalarized} in the variables $(\x, \y)$.
% Interestingly, even when \eqref{eq::notscalarized} has only the trivial sign symmetry in sense of 
% Definition \eqref{def::ss}, the scalarized problem \eqref{eq::scalarized} can reveal some
% ``hidden'' nontrivial sign symmetries. We illustrate this phenomenon in the following example.
% \end{rem}

\begin{exmp}
Consider the problem
\begin{equation}\label{sec4:ex1}
\inf\limits_{\x\in\R^2}\lambda_{\min}(F(\x))\quad\mathrm{s.t.}\quad 1 - x_1^2 - x_2^2\ge 0
\end{equation}
with
\begin{equation}
F(\x) = \begin{bmatrix}x_1^2&x_1+x_2\\ 
x_1+x_2& x_2^2\end{bmatrix}.
\end{equation}
Note that $F(\x)$ has a nontrivial PMI sign symmetry $\theta=(-1,-1)$.
% Note that \eqref{sec4:ex1} has only the trivial sign symmetry \rev{in sense of 
% Definition \eqref{def::ss}}, 
% which means that the related block structure is also trivial. 
With the relaxation order $r=2$, the procedure for exploiting term sparsity with block closures stabilizes at $s=2$, giving rise to two blocks for the moment matrix: 
\begin{equation}\label{eq::partition1}
\{[1]^1,[x_1]^2,[x_2]^2, [x_1^2]^1,[x_1x_2]^1,[x_2^2]^1\},
\qquad \{[1]^2,[x_1]^1,[x_2]^1, [x_1^2]^2,[x_1x_2]^2,[x_2^2]^2\}.
\end{equation}
The corresponding sparse relaxation \eqref{eq::mPMI} yields a lower bound $-0.9142$ which coincides with the bound given by the second order dense relaxation.

% \rev{Now consider the scalarized problem
% \begin{equation}\label{sec4:ex1_1}
% \left\{
% \begin{aligned}
% \inf\limits_{\x\in\R^2, \y\in\R^2}&
% \y^\intercal F(\x) \y=x_1^2y_1^2+x_2^2y_2^2+(x_1+x_2)y_1y_2\\
% \mathrm{s.t.}\ &1 - x_1^2 - x_2^2\ge 0,\quad 1-y_1^2-y_2^2=0.
% \end{aligned}\right.
% \end{equation}
% Note that \eqref{sec4:ex1_1} exhibits nontrivial sign symmetries in $(\x, \y)$,  
% represented by the binary matrix 
% \[
% R=\begin{bmatrix} 1 & 1 & 1 & 0\\
% 1 & 1 & 0 & 1\\
% \end{bmatrix}^{\intercal}.
% \]
% By these sign symmetries in $(\x, \y)$, the monomial subset 
% \begin{equation*}\label{eq::mono}
% \{[1]^1y_1, [x_1]^1y_1, [x_2]^1y_1, [x_1^2]^1y_1, [x_1x_2]^1y_1, [x_2^2]^1y_1, [1]^2y_2, [x_1]^2y_2, [x_2]^2y_2, [x_1^2]^2y_2, [x_1x_2]^2y_2, [x_2^2]^2y_2\}.
% \end{equation*}
% used in the moment relaxation of order $2$ for \eqref{sec4:ex1_1} 
% gives rise to two blocks for the moment matrix
% \begin{equation}\label{eq::partition2}
% \{[1]^1y_1, [x_1^2]^1y_1, [x_1x_2]^1y_1, [x_2^2]^1y_1, [x_1]^2y_2, [x_2]^2y_2\},\
% \{[x_1]^1y_1, [x_2]^1y_1, [1]^2y_2, [x_1^2]^2y_2, [x_1x_2]^2y_2, [x_2^2]^2y_2\}.
% \end{equation}
% By removing $y_1$ and $y_2$ in \eqref{eq::partition2}, we exactly recover the partition in \eqref{eq::partition1}. 
% This demonstrates that the block structures of the
% term sparsity iterations for \eqref{sec4:ex1} indeed converge to the one determined by the latent 
% sign symmetries uncovered by the scalarization \eqref{sec4:ex1_1}.
% }
\end{exmp}

\section{Correlative Sparsity}
\label{sec:correlative}

The aim of this section is to explore the correlative sparsity of the PMO problem
\eqref{eq::mPMI}:
\begin{equation*}
\lambda^{\star}\coloneqq\inf_{\x\in\R^n}\lambda_{\min}(F(\x))\quad \text{s.t. }\  G_1(\x)\succeq 0,\ldots,G_m(\x)\succeq 0, 
\end{equation*}
where $F\in\mS^p[\x]$ and $G_k\in\mS^{q_k}[\x], k\in[m]$.

We define the {\em correlative sparsity pattern (CSP) graph} of \eqref{eq::mPMI}, denoted by $\G^{\mathrm{csp}}$, such that $\vs(\G^{\mathrm{csp}})=[n]$ and $\{i,j\}\in \es(\G^{\mathrm{csp}})$ if one of following conditions holds:
\begin{enumerate}
    \item[(i)] There exists $\a\in\supp{F}$ such that $i,j\in\supp{\a}\coloneqq\{k\in[n]\mid\alpha_k\ne0\}$;
    \item[(ii)] There exists $k\in[m]$ such that $x_i,x_j\in\var(G_k)$, where $\var(G_k)$ is the subset of variables effectively involved in $G_k$.
\end{enumerate}

Let $\{\mI_\ell\}_{\ell=1}^t$ be the list of maximal cliques of $\G^{\mathrm{csp}}$ with $n_\ell\coloneqq|\mI_\ell|$. 
% Then, after a possible rearrangement, we may assume that 
% $\{\mI_\ell\}_{\ell=1}^t$ satisfies the running intersection property (RIP), 
% i.e., for every $\ell\in[t-1]$, there exists $s\in [\ell]$ such that
% \[
% \mI_{\ell+1}\cap \bigcup_{j=1}^{\ell} \mI_j\subseteq \mI_s.
% \]
For $\ell\in[t]$, let $\R[\x(\mI_\ell)]$ denote the ring of polynomials in the $n_\ell$ variables $\x(\mI_\ell) \coloneqq\{x_i\mid i\in \mI_\ell\}$. 
We can partition the constraint polynomial matrices $G_1,\ldots,G_m$ into groups 
$\{G_k\mid k\in \mJ_\ell\}, \ell=1,\ldots,t$ which satisfy
\begin{enumerate}
    \item[(i)] $\mJ_1,\ldots,\mJ_t\subseteq[m]$ are pairwise disjoint and $\cup_{\ell=1}^t\mJ_\ell=[m]$;
    \item[(ii)] For any $k\in \mJ_\ell$, $\var(G_k)\subseteq \mI_\ell$, $\ell=1,\ldots,t$.
\end{enumerate}

Given a sequence $\bS=(S_{\a})_{\a\in\N^n}\subseteq\mS^p$ and $\ell\in[t]$, for $d\in\N$ and $G\in\R[\x(\mI_\ell)]$, 
let $M_d(\bS, \mI_\ell)$ (resp. $M_d(G\bS, \mI_\ell)$)
be the moment (resp. localizing) submatrix obtained from $M_d(\bS)$ (resp. $M_d(G\bS)$) 
by retaining only those block rows and columns indexed by $\a\in \N^n_d$
with $\supp{\a}\subseteq \mI_\ell$.

Then with $r\ge r_{\min}$, the $r$-th order correlative sparsity adapted moment relaxation for \eqref{eq::mPMI} is given by
\begin{equation}\label{eq::mPMI::csmom}
\lambda^{(\cs)}_{r}\coloneqq
\begin{cases}
\inf &\mL_{\bS}(F)\\
\mathrm{s.t.}&M_{r}(\bS, \mI_\ell)\succeq0,\quad \ell\in[t],\\
&M_{r-d_k}(G_k\bS, \mI_\ell)\succeq0,\quad k\in \mJ_\ell, \ell\in[t],\\
&\mL_{\bS}(I_p)=1,
\end{cases}
\end{equation}
and its dual problem reads as 
\begin{equation}\label{eq::mPMI::cssos}
    \begin{cases}\sup&\lambda \\
    \mathrm{s.t.} &F-\lambda I_p=\sum_{\ell=1}^t\left(S_{\ell,0}
    +\sum_{k\in \mJ_\ell}\left\langle S_{\ell,k}, G_k\right\rangle_{p}\right),\\
    &S_{\ell,0}\in\Sigma^p[\x(\mI_{\ell})]_{2r},\ S_{\ell,k}\in\Sigma^{pq_k}[\x(\mI_\ell)]_{2(r-d_k)},\quad k\in \mJ_\ell,\ell\in[t].
    \end{cases}
\end{equation}
Each $\lambda^{(\cs)}_r$ therefore provides a lower bound on $\lambda^{\star}$ and
the sequence $\left\{\lambda^{(\cs)}_r\right\}_{r\ge r_{\min}}$ is monotonically non-decreasing. Moreover, asymptotic convergence to $\lambda^{\star}$ can be guaranteed under appropriate conditions when $p=1$.

\begin{thm}{\rm (\cite[Theorem 1]{KM2009})}\label{sec5:thm0}
Let $p=1$. Assume that the following conditions hold:
\begin{enumerate}
   \item[\upshape (i)] The CSP $\{\mI_\ell\}_{\ell=1}^t$ satisfies the running intersection property, i.e., for every $\ell\in[t-1]$, there exists $s\in [\ell]$ such that
   \[\mI_{\ell+1}\cap \bigcup_{j=1}^{\ell} \mI_j\subseteq \mI_s;\]
   \item[\upshape (ii)] For each $\ell\in [t]$, there exists $a_\ell>0$ such that 
   \[
   a_\ell-\sum_{i\in \mI_\ell} x_i^2= \sigma_{\ell,0}
   +\sum_{k\in \mJ_\ell}\langle S_{\ell,k}, G_k\rangle,
   \]
   where $\sigma_{\ell,0}\in\Sigma^1[\x],\ S_{\ell,k}\in\Sigma^{q_k}[\x]$.
\end{enumerate}
Then $\lim_{r\to\infty}\lambda^{(\cs)}_r=\lambda^{\star}$.
\end{thm}

% It is known \cite{FKMN2001,NFFKM2003} that after some possible reordering of the cliques, 
% $\{\mI_\ell\}_{\ell=1}^t$ satisfies the 
% {\itshape running intersection property} (RIP).
% \begin{prop} The following statements are true:
% \begin{enumerate}
%     \item[\upshape (ii)]
%    If $p=1$ and Assumption \ref{assump2} holds true, then $\lim_{r\to\infty}\lambda^{(\cs)}_r=\lambda^{\star}$.
% \end{enumerate}
% \end{prop}
% \begin{proof}
% (i) Raising $r$ leads to additional constraints posed over the psudomoments in the moment matrix LMI, thus ensuring monotonicity in objective.

% (ii) It follows from \cite[Theorem 1]{KM2009} and the weak duality between \eqref{eq::mPMI::csmom} and 
% \eqref{eq::mPMI::cssos}.
% \end{proof}

In view of Theorem \ref{sec5:thm0}, one may expect that the asymptotic convergence also holds true when $p>1$ under similar conditions.
However, the following counterexample (inspired by \cite{ning2014matrix}) shows that this is not the case even when $p=2$.
In other words, the matrix counterpart (i.e., for $p>1$) of Putinar's Positivstellensatz for polynomials 
with correlative sparsity studied in \cite{KM2009,LasserreSparsity,GNS2007} does not hold in general. 
 
\begin{exmp} 
Let 
\begin{equation*} 
F(\x)=\begin{bmatrix}
2(x_1-1)^2+(x_2-1)^2+(x_2-2)^2& 3-2x_2\\ 3-2x_2&2(x_1-2)^2+(x_2-1)^2+(x_2-2)^2
\end{bmatrix} \end{equation*} 
and 
\[
\K=\left\{(x_1, x_2)\in\R^2 : 4-x_1^2\ge0, \, 4-x_2^2\ge0\right\}.
\]
Let 
\[
Q=\begin{bmatrix}
\frac{13}{2} & 3 &  -2 &  -3   &-3    &2\\
    3  & \frac{25}{2}    &3   &-4   &-4 &  -3\\
   -2 &   3 &   2&    0 &   0 &  -2\\
   -3 &  -4 &   0&    2 &   2 &   0\\
   -3 &  -4 &   0&    2 &   2 &   0\\
    2 &  -3 &  -2&    0 &   0 &   2
\end{bmatrix}.
\]
The matrix $Q$ has three eigenvalues: $0$ and $\frac{27}{2}\pm 3\sqrt{2}$ (so $Q\succeq 0$), and 
\[
F(\x)-\frac{1}{2}I_2= (\bm_1(\x)\otimes I_2)^{\intercal}Q(\bm_1(\x)\otimes I_2)\succeq 0.
\]
Moreover, we have $(\frac{3}{2}, 1)\in \K$ and $\lambda_{\min}\left(F(\frac{3}{2}, 1)\right)=\frac{1}{2}$, so $\inf_{\x\in\K}\lambda_{\min}(F(\x))=\frac{1}{2}$ (whose minimizers form a circle: $(x_1-\frac{3}{2})^2+(x_2-\frac{3}{2})^2=\frac{1}{4}$).
Hence, for any $\varepsilon>0$, 
\[
F(\x)-\left(\frac{1}{2}-\varepsilon\right)I_2\succ 0,\quad \forall \x\in\K.
\]
However, for any $\varepsilon\in(0, \frac{1}{2})$, we will show that the representation
\begin{equation}\label{eq::smp}
F(\x)-\left(\frac{1}{2}-\varepsilon\right)I_2=\sum_{i=1}^2\left(S_{i,0}(x_i)+(4-x_i^2)S_{i,1}(x_i)\right),
\end{equation}
where each $S_{i,j}\in\Sigma^2[x_i]$, does not hold.
Suppose on the contrary that there exists a representation \eqref{eq::smp} for some $\varepsilon\in(0, \frac{1}{2})$.
Let $\Phi_1,\Phi_2$ be the following matrix-valued atomic measures defined in the $x_1$-space and $x_2$-space, respectively:
\[
\Phi_1=\begin{bmatrix}
1 & 0\\ 0 & 0
\end{bmatrix}\delta_{(x_1=1)}+
\begin{bmatrix}
0 & 0\\ 0 & 1
\end{bmatrix}\delta_{(x_1=2)},\quad
\Phi_2=\begin{bmatrix}
\frac{1}{2} & -\frac{1}{2}\\ -\frac{1}{2} & \frac{1}{2}
\end{bmatrix}\delta_{(x_2=1)}+
\begin{bmatrix}
\frac{1}{2} & \frac{1}{2}\\ \frac{1}{2} & \frac{1}{2}
\end{bmatrix}\delta_{(x_2=2)}.
\]
% Let $2d$ be the maximal degree of each term in the right hand side of \eqref{eq::smp}
% and consider the following subcone of $\Sigma^2[\x]_{2d}$
% \[
% \mathcal{C}\coloneqq\sum_{i=1}^2\left(S_{i,0}(x_i)+(4-x_i^2)S_{i,1}(x_i)\right).
% \]
Then, we define a linear functional $\mL : \mS^2[x_1]+\mS^2[x_2] \to \R$ by
\[
\mL\left(H_1(x_1)+H_2(x_2)\right)=
\mL_{\Phi_1}(H_1(x_1))+\mL_{\Phi_2}(H_2(x_2)), \quad\forall H_1(x_1)\in \mS^2[x_1], H_2(x_2)\in \mS^2[x_2].
\]
Since $\mL_{\Phi_1}(I_2)=\mL_{\Phi_2}(I_2)=1$, there is no ambiguity in the above definition. 
One can easily check that $\mL(F)=0$. Now applying $\mL$ to both sides
of \eqref{eq::smp}, we obtain
\[
0-\left(\frac{1}{2}-\varepsilon\right)=\mL\left(F(\x)-\left(\frac{1}{2}-\varepsilon\right)I_2\right)
=\mL\left(\sum_{i=1}^2\left(S_{i,0}(x_i)+(4-x_i^2)S_{i,1}(x_i)\right)\right)\ge 0,
\]
which is a contradiction.
\end{exmp}

\rev{
\begin{rem}
We are not aware of any simple condition that can guarantee asymptotic convergence of the correlative sparsity adapted hierarchy for PMO with a matrix objective. This will be pursued in our future research. On the other hand, one may restore asymptotic convergence in the case of correlative sparsity by considering the scalarized problem $\inf_{\x\in\K, \|\y\|^2=1} \y^\intercal F(\x)\y$.
\end{rem}}

Even though asymptotic convergence does not hold for PMO with correlative sparsity in general, finite convergence may still occur in practice, and this convergence could be detected by flatness conditions. For the rest of this section, we establish several results on detecting finite convergence and extracting optimal solutions when exploiting correlative sparsity. For each $\ell\in[t]$, let us define
\begin{equation}
    \K_\ell\coloneqq\left\{\x(\ell)\in\R^{|\mI_{\ell}|}\,\middle|\, G_k(\x)\succeq0,\quad \forall k\in \mJ_\ell\right\},
\end{equation}
where $\x(\ell)\coloneqq(x_i)_{i\in\mI_{\ell}}$.

\begin{thm}\label{sec5:thm1}
Let $\bS$ be an optimal solution  
of \eqref{eq::mPMI::csmom}. Assume that the following conditions hold:
\begin{enumerate}
%    \item[\upshape (i)] There exists a positive integer $s$ such that for each $\ell\in[t]$, one has 
%    \begin{equation}
%        \rank{M_{r}(\bS, \mI_\ell)}=\rank{M_{r-d^{(\ell)}}(\bS, \mI_\ell)}.
%    \end{equation}
    \item[\upshape (i)] For each $\ell\in[t]$, $\bS^{\ell}\coloneqq(S_{\a})_{\supp{\a}\subseteq\mI_\ell}$
%the matrix measure $\Phi_{\ell}$ representing the sequence 
%$\bS^{\ell}\coloneqq(S_{\a})_{\supp{\a}\subseteq\mI_\ell}$ 
admits a representing measure
%decomposes as 
$\Phi_{\ell}=S_{\mathbf{0}}\left(\sum_{i=1}^{s_{\ell}}\xi_{i\ell}\delta_{\x^{i}(\ell)}\right)$ for some $\xi_{i\ell}>0$ with $\sum_{i=1}^{s_{\ell}}\xi_{i\ell} = 1$ and $\x^{1}(\ell),\ldots,\x^{s_{\ell}}(\ell)\in\K_{\ell}$;
    \item[\upshape (ii)] For all pairs $\{i,j\}$ with $\mI_{i}\cap\mI_{j}\ne\emptyset$, one has
     $\rank{M_{r}(\bS, \mI_{i}\cap\mI_{j})}=1$.
\end{enumerate}
Then $\lambda_{r}^{(\cs)}=\lambda^{\star}$. Moreover, each point $\hat{\x}\in\R^n$ satisfying $\hat{\x}(\ell)\coloneqq(\hat{x}_i)_{i\in\mI_\ell}=\x^{i_{\ell}}(\ell)$ for some $i_{\ell}\in[s_{\ell}]$, $\ell\in[t]$, is an optimal solution of \eqref{eq::mPMI}.
\end{thm}
\begin{proof}
For each $\ell\in[t]$, let us pick a point $\x^{i_{\ell}}(\ell)$, $i_{\ell}\in[s_{\ell}]$, and then define the point $\hat{\x}\in\R^n$ such that $\hat{\x}(\ell)\coloneqq(\hat{x})_{i\in\mI_\ell}=\x^{i_{\ell}}(\ell)$, $\ell\in[t]$. Because $\rank{M_{r}(\bS, \mI_{i}\cap\mI_{j})}=1$ for all pairs $\{i,j\}$ with $\mI_{i}\cap\mI_{j}\ne\emptyset$, the value of $\hat{x}_k$ is unique for any $k\in\mI_{i}\cap\mI_{j}\ne\emptyset$. Therefore, $\hat{\x}$ is well-defined and $\hat{\x}\in\K$.
We can thus construct $s=\prod_{j=1}^{\ell}s_{\ell}$ solutions $\{\x^j\}_{j=1}^{s}\subseteq\K$, 
each associated with the weight matrix $W_{j}\coloneqq\prod_{\ell=1}^t\xi_{i_{\ell}\ell}S_{\mathbf{0}}\in\mS^p_+$ if $\hat{\x}(\ell)=\x^{i_{\ell}}(\ell)$ for some $i_{\ell}\in[s_{\ell}]$. We then define the following matrix measure
$\Phi=\sum_{j=1}^{s}W_{j}\delta_{\x^{j}}$
which is supported on $\K$, and its marginal measure on $\K_{\ell}$ is $\Phi_{\ell}$ for each $\ell\in[t]$. 
Moreover, it holds that
\begin{equation}
\sum_{j=1}^sW_j=\prod_{\ell=1}^t\left(\sum_{i=1}^{s_{\ell}}\xi_{i\ell}\right)S_{\mathbf{0}}=S_{\mathbf{0}}.
\end{equation}
Therefore, we have
\begin{equation}
    \lambda^{\star}\ge\lambda_{r}^{(\cs)}=\mL_{\bS}(F)=\sum_{j=1}^{s}\left\langle W_{j}, F(\x^{j})\right\rangle\ge\sum_{j=1}^{s}\left\langle W_{j}, \lambda^{\star}I_p\right\rangle=\lambda^{\star}\left\langle \sum_{j=1}^{s}W_{j}, I_p\right\rangle=\lambda^{\star}\tr{S_{\mathbf{0}}}=\lambda^{\star},
\end{equation}
which implies $\lambda_{r}^{(\cs)}=\lambda^{\star}$ and each $\x^{j}$ is an optimal solution of \eqref{eq::mPMI}.
\end{proof}

\begin{exmp}
Consider the PMO problem
\begin{equation}\label{sec5:eq2}
\inf_{\x\in\R^3}\lambda_{\min}\left((-x_1^2 + x_2)(\bq_1\bq_1^{\intercal}+\bq_2\bq_2^{\intercal}) + (x_2^2 + x_3^2)\bq_3\bq_3^{\intercal}\right)\quad\mathrm{s.t.}\quad 1 - x_1^2 - x_2^2\ge 0,\, x_2^2 + x_3^2 = 1,
\end{equation}
where $\bq_1,\bq_2,\bq_3$ are the column vectors of the orthogonal matrix
\begin{equation}
Q = \begin{bmatrix} \frac{1}{\sqrt{2}}  & -\frac{1}{\sqrt{3}} &  \frac{1}{\sqrt{6}}\\
  0  &  \frac{1}{\sqrt{3}}  & \frac{2}{\sqrt{6}}\\
  \frac{1}{\sqrt{2}}  & \frac{1}{\sqrt{3}}  &  -\frac{1}{\sqrt{6}}\end{bmatrix}.
\end{equation}
Problem \eqref{sec5:eq2} exhibits a CSP: $\mI_1=\{1,2\}$ and $\mI_2=\{2,3\}$. Solving \eqref{eq::mPMI::csmom} with $r=2$, we obtain $\lambda_{2}^{(\cs)}=-1.250$ and $\rank{M_{2}(\bS, \mI_1)}=\rank{M_{1}(\bS, \mI_1)}=4$, $\rank{M_{2}(\bS, \mI_2)}=\rank{M_{1}(\bS, \mI_2)}=4$. Applying the extraction procedure in \cite{GW2023} to $M_{2}(\bS, \mI_1)$ and $M_{2}(\bS, \mI_2)$, respectively, we retrieve the matrix measures:
\begin{equation}
\Phi_1\approx W\left(\delta_{\x^{1}(1)}+\delta_{\x^{2}(1)}\right)\text{ and }\Phi_2\approx W\left(\delta_{\x^{1}(2)}+\delta_{\x^{2}(2)}\right)
\end{equation}
with
\begin{equation}
W\approx\begin{bmatrix}0.2083&-0.0833&0.0416\\-0.0833&0.0833&0.0833\\0.0416&0.0833&0.2083 \end{bmatrix}
\end{equation}
and
\begin{equation}
    \x^{1}(1)\approx\begin{bmatrix} 0.8660\\-0.4999\end{bmatrix},\x^{2}(1)\approx\begin{bmatrix}-0.8660\\-0.4999\end{bmatrix},\x^{1}(2)\approx\begin{bmatrix} -0.4999\\0.8660\end{bmatrix},\x^{2}(2)\approx\begin{bmatrix} -0.4999\\-0.8660\end{bmatrix}.
\end{equation}
We can merge $\Phi_1$ and $\Phi_2$ into
\begin{equation}
\Phi\approx W\left(\delta_{\hat{\x}^{1}}+\delta_{\hat{\x}^{2}}+\delta_{\hat{\x}^{3}}+\delta_{\hat{\x}^{4}}\right)
\end{equation}
with
\begin{equation}
    \hat{\x}^{1}\approx\begin{bmatrix} 0.8660\\-0.4999\\0.8660\end{bmatrix},\hat{\x}^{2}\approx\begin{bmatrix} 0.8660\\-0.4999\\-0.8660\end{bmatrix},\hat{\x}^{3}\approx\begin{bmatrix} -0.8660\\-0.4999\\0.8660\end{bmatrix},\hat{\x}^{4}\approx\begin{bmatrix} -0.8660\\-0.4999\\-0.8660\end{bmatrix}.
\end{equation}
Thus, by Theorem \eqref{sec5:thm1}, the optimum of \eqref{sec5:eq2} is $-1.250$ which is achieved at $\hat{\x}^{1}$, $\hat{\x}^{2}$, $\hat{\x}^{3}$, and $\hat{\x}^{4}$.
\end{exmp}

For each $\ell\in[t]$, let $d_{\K_{\ell}}\coloneqq\max_{j\in \mJ_{\ell}}\lceil\deg{G_j}/2\rceil$.
\begin{cor}
Let $\bS$ be an optimal solution  
of \eqref{eq::mPMI::csmom}. Assume that the following conditions hold:
\begin{enumerate}
    \item[\upshape (i)] $\rank{S_{\mathbf{0}}}=1$;
%     \item[\upshape (ii)] For each $\ell\in[t]$, $\bS^{\ell}$ admits the representing matrix measure
% $\Phi_{\ell}=\sum_{i=1}^{s_{\ell}}W_{i\ell}\delta_{\x^{i}(\ell)}$ for some 
% $W_{i\ell}\in\mathbb{S}_+^p$, $\x^{1}(\ell),\ldots,\x^{s_{\ell}}(\ell)\in\R^n$.
   \item[\upshape (ii)] For each $\ell\in[t]$, one has
   \begin{equation}\label{sec5:eq1}
       \rank{M_{r}(\bS, \mI_\ell)}=\rank{M_{r-d_{\K_{\ell}}}(\bS, \mI_\ell)}(\eqqcolon s_{\ell});
   \end{equation}
    \item[\upshape (iii)] For all pairs $\{i,j\}$ with $\mI_{i}\cap\mI_{j}\ne\emptyset$, one has
     $\rank{M_{r}(\bS, \mI_{i}\cap\mI_{j})}=1$.
\end{enumerate}
Then $\lambda_{r}^{(\cs)}=\lambda^{\star}$. 
Moreover, let $\Delta_{\ell}\coloneqq\{\x^{i}(\ell)\}_{i=1}^{s_{\ell}}$ be the set of points obtained by applying the extraction procedure in \cite{GW2023} to each moment matrix $M_{r}(\bS, \mI_\ell)$, $\ell\in[t]$, and 
let $S_{\mathbf{0}}=\bv\bv^{\intercal}$ for some $\bv\in\R^p$. 
Then each point $\hat{\x}\in\R^n$ satisfying $\hat{\x}(\ell)\coloneqq(\hat{x}_i)_{i\in\mI_\ell}=\x^{i_{\ell}}(\ell)$ for some $\x^{i_{\ell}}(\ell)\in\Delta_{\ell}$, $\ell\in[t]$, 
is an optimal solution of \eqref{eq::mPMI} and $\bv$ is the corresponding eigenvector.
\end{cor}
\begin{proof}
% First of all, we observe that $\rank{S_{\mathbf{0}}}=1$ and $\tr{S_{\mathbf{0}}}=1$ (for $\mL_{\bS}(I_p)=1$) imply that $S_{\mathbf{0}}^t=S_{\mathbf{0}}$.
% Let $s_{\ell}\coloneqq\rank{M_{r}(\bS, \mI_\ell)}, \ell\in[t]$. 
By \eqref{sec5:eq1} and Theorem \ref{th::FEC}, for each $\ell\in[t]$, the sequence of matrices $\bS^{\ell}\coloneqq(S_{\a})_{|\a|\le 2r,\supp{\a}\subseteq\mI_\ell}$ admits an $s_{\ell}$-atomic matrix measure $\Phi_{\ell}$ supported on $\K_{\ell}$, $\ell\in[t]$ so that
\begin{equation}  \Phi_{\ell}=\sum_{i=1}^{s_{\ell}}W_{i\ell}\delta_{\x^{i}(\ell)} \text{ for some } W_{i\ell}\in\mS^p_+.
\end{equation}
Note that $\sum_{i=1}^{s_{\ell}}W_{i\ell} = S_{\mathbf{0}}$ for each $\ell\in[t]$.
The fact that $\rank{S_{\mathbf{0}}}=1$ together with $W_{i\ell}\succeq0$ implies that $\rank{W_{i\ell}}=1$ for all $i,\ell$, from which we deduce that there exists $\xi_{i\ell}>0$ such that $W_{i\ell}=\xi_{i\ell}S_{\mathbf{0}}$ for all $i,\ell$ and $\sum_{i=1}^{s_{\ell}}\xi_{i\ell} = 1,\ell\in[t]$. 
Then the conclusion follows from Theorem \ref{sec5:thm1}.
\end{proof}

\begin{thm}\label{sec5:thm3}
Let $\bS$ be an optimal solution  
of \eqref{eq::mPMI::csmom}. Assume that the following conditions hold:
\begin{enumerate}
%    \item[\upshape (i)] There exists a positive integer $s$ such that for each $\ell\in[t]$, one has 
%    \begin{equation}
%        \rank{M_{r}(\bS, \mI_\ell)}=\rank{M_{r-d^{(\ell)}}(\bS, \mI_\ell)}=s.
%    \end{equation}
    \item[\upshape (i)] There exist weight matrices $W_1,\ldots,W_s\in\mS^p_+$ such that for each $\ell\in[t]$, 
%the matrix measure $\Phi_{\ell}$ representing the sequence $\bS^{\ell}\coloneqq(S_{\a})_{\supp{\a}\subseteq\mI_\ell}$ decomposes as 
$\bS^{\ell}\coloneqq(S_{\a})_{\supp{\a}\subseteq\mI_\ell}$ admits a representing matrix measure 
$\Phi_{\ell}=\sum_{i=1}^{s}W_{i}\delta_{\x^{i}(\ell)}$ for some points $\x^{1}(\ell),\ldots,\x^{s}(\ell)\subseteq\K_{\ell}$;
    \item[\upshape (ii)] For each $i\in[s]$, the points $\x^{i}(1),\ldots,\x^{i}(t)$ can be merged into a single point $\hat{\x}^{i}\in\R^n$ such that $\hat{\x}^{i}(\ell)=\x^{i}(\ell)$ for every $\ell\in[t]$.
\end{enumerate}
Then $\lambda_{r}^{(\cs)}=\lambda^{\star}$. Moreover, each point $\hat{\x}^{i},i\in[s]$ is an optimal solution of \eqref{eq::mPMI}.
\end{thm}
\begin{proof}
Let us define the matrix measure $\Phi=\sum_{i=1}^{s}W_{i}\delta_{\hat{\x}^{i}}$
%\begin{equation}
%    \Phi=\sum_{i=1}^{s}W_{i}\delta_{\hat{\x}^{i}}
%\end{equation}
which is supported on $\K$, and its marginal measure on $\K_{\ell}$ is $\Phi_{\ell}$ for each $\ell\in[t]$. Then,
\begin{equation}
    \lambda^{\star}\ge\lambda_{r}^{(\cs)}=\mL_{\bS}(F)=\sum_{i=1}^{s}\left\langle W_{i}, F(\hat{\x}^{i})\right\rangle\ge\sum_{i=1}^{s}\left\langle W_{i}, \lambda^{\star}I_p\right\rangle=\lambda^{\star}\left\langle \sum_{i=1}^{s}W_{i}, I_p\right\rangle=\lambda^{\star}\tr{S_{\mathbf{0}}}=\lambda^{\star},
\end{equation}
which implies $\lambda_{r}^{(\cs)}=\lambda^{\star}$ and each $\hat{\x}^{i}$ is an optimal solution of \eqref{eq::mPMI}.
\end{proof}

\begin{exmp}
Consider the PMO problem
\begin{equation}\label{cs:ex2}
\inf_{\x\in\R^3}\lambda_{\min}(F(\x))\quad\mathrm{s.t.}\quad 1 - x_1^2 - x_2^2\ge 0,\, 1 - x_2^2 - x_3^2\ge 0
\end{equation}
with
\begin{equation}
F(\x) = \begin{bmatrix}x_1^2 + x_2^2 &2 + x_1x_2 + x_3^2\\ 
2 + x_1x_2 + x_3^2& x_2x_3\end{bmatrix},
\end{equation}
which exhibits a CSP: $\mI_1=\{1,2\}$ and $\mI_2=\{2,3\}$. Solving \eqref{eq::mPMI::csmom} with $r=2$, we obtain $\lambda_{2}^{(\cs)}\approx-3.0643$ and $\rank{M_{2}(\bS, \mI_1)}=\rank{M_{1}(\bS, \mI_1)}=2$, $\rank{M_{2}(\bS, \mI_2)}=\rank{M_{1}(\bS, \mI_2)}=2$. Applying the extraction procedure in \cite{GW2023} to $M_{2}(\bS, \mI_1)$ and $M_{2}(\bS, \mI_2)$, respectively, we retrieve the matrix measures:
\begin{equation}
    \Phi_1\approx\begin{bmatrix}0.2338&-0.2494\\-0.2494&0.2661 \end{bmatrix}\delta_{\x^{1}(1)}+\begin{bmatrix}0.2338&-0.2494\\-0.2494&0.2661 \end{bmatrix}\delta_{\x^{2}(1)},
\end{equation}
\begin{equation}
    \Phi_2\approx\begin{bmatrix}0.2338&-0.2494\\-0.2494&0.2661 \end{bmatrix}\delta_{\x^{1}(2)}+\begin{bmatrix}0.2338&-0.2494\\-0.2494&0.2661 \end{bmatrix}\delta_{\x^{2}(2)}
\end{equation}
with
\begin{equation}
    \x^{1}(1)\approx\begin{bmatrix} 0.2732\\0.2561\end{bmatrix},\x^{2}(1)\approx\begin{bmatrix} -0.2732\\-0.2561\end{bmatrix},\x^{1}(2)\approx\begin{bmatrix} 0.2561\\-0.9666\end{bmatrix},\x^{2}(2)\approx\begin{bmatrix} -0.2561\\0.9666\end{bmatrix}.
\end{equation}
We can merge $\Phi_1$ and $\Phi_2$ into
\begin{equation}
\Phi\approx\begin{bmatrix}0.2338&-0.2494\\-0.2494&0.2661 \end{bmatrix}\delta_{\hat{\x}^{1}}+\begin{bmatrix}0.2338&-0.2494\\-0.2494&0.2661 \end{bmatrix}\delta_{\hat{\x}^{2}}
\end{equation}
with
\begin{equation}
    \hat{\x}^{1}\approx\begin{bmatrix} 0.2732\\0.2561\\-0.9666\end{bmatrix},\hat{\x}^{2}\approx\begin{bmatrix} -0.2732\\-0.2561\\0.9666\end{bmatrix}.
\end{equation}
Thus, by Theorem \eqref{sec5:thm3}, the optimum of \eqref{cs:ex2} is $-3.0643$ which is achieved at $\hat{\x}^{1}$ and $\hat{\x}^{2}$.
\end{exmp}

\begin{cor}
Let $\bS$ be an optimal solution  
of \eqref{eq::mPMI::csmom}.
If $\rank{M_{r_{\min}}(\bS, \mI_\ell)}=1$ for each $\ell\in[t]$, then $\lambda_{r}^{(\cs)}=\lambda^{\star}$. Moreover, one can recover an optimal solution of \eqref{eq::mPMI}.
\end{cor}
\begin{proof}
Since $\rank{M_{r_{\min}}(\bS, \mI_\ell)}=1$, by \cite[Theorem 5]{GW2023}, each sequence $\bS^{\ell}\coloneqq(S_{\a})_{|\a|\le 2r_{\min},\supp{\a}\subseteq\mI_\ell}$ admits a Dirac representing measure $\Phi_{\ell}=S_{\mathbf{0}}\delta_{\x(\ell)}$ 
with $\x(\ell)\in\K_{\ell}$.
Moreover, we can merge $\x(1),\ldots,\x(t)$ into a point $\hat{\x}\in\R^n$ by letting
\[
\hat{\x}(\mI_{\ell})\coloneqq(\hat{x}_i)_{i\in\mI_{\ell}}=\x(\ell),\quad \forall\ell\in[t].
\]
There is no ambiguity for $\hat{x}_i$ when $i\in\mI_j\cap\mI_k\neq\emptyset$ for some $j, k\in[t]$.
In fact, denoting by $\be_i\in\N^n$ the vector whose $i$-th entry being $1$ and the others 
being $0$, we have $S_{\be_i}=S_{\mathbf{0}}\x(j)_i=S_{\mathbf{0}}\x(k)_i$, which implies $\x(j)_i=\x(k)_i$. 
Then, the conclusion follows from Theorem \ref{sec5:thm3}.
%Moreover, it holds that $\hat{\x}\in\K$.
%We define the matrix measure $\Phi\coloneqq S_{\mathbf{0}}\delta_{\hat{\x}}$ which, by construction, is the representing measure of $(S_{\a})_{|\a|\le 2r_{\min}}$. Assume $S_{\mathbf{0}}=\bv\bv^{\intercal}$ for some $\bv\in\R^p$ with
%$\Vert\bv\Vert_2=1$. We have
%\[
%\lambda^{\star}\ge\lambda_{r}^{(\cs)}=\mL_{\bS}(F)=\langle F(\hat{\x}), S_{\mathbf{0}}\rangle=\bv^{\intercal} F(\hat{\x})\bv\ge
%\lambda^{\star}.
%\]
%Therefore, $\lambda_{r}^{(\cs)}=\lambda^{\star}$ and $\hat{\x}$ is an optimal solution of \eqref{eq::mPMI} with $\bv$ being the corresponding eigenvector.
\end{proof}

% \vspace{1em}
% Besides term sparsity and correlative sparsity, a PMO problem may also possess matrix sparsity.
% % In the online appendix\footnote{The appendix is available at \url{https://arxiv.org/abs/2411.15479}.},
% In the appendix, 
% we describe how to exploit matrix sparsity for PMO.
\rev{\section{Matrix Sparsity}\label{cms}
In this section, we briefly introduce how to exploit matrix sparsity (i.e., chordal sparsity pattern of objective or constraint polynomial matrices) to perform reductions on the SDP size for PMO \eqref{eq::mPMI}.
The cases for objective polynomial matrices and for constraint polynomial matrices will be addressed sequentially.

\subsection{Objective matrix sparsity}
To exploit chordal sparsity encoded in the objective matrix, we adapt Theorem 2.4, Theorem 2.5, and Corollary 2.2 in \cite{zheng2023sum} dealing with scalar polynomial constraints to the case of PMI constraints. For the proofs, one only needs to replace the usual inner product $\langle\cdot,\cdot\rangle$ in the corresponding proofs in \cite{zheng2023sum} with the product $\langle\cdot,\cdot\rangle_p$. 
We thus omit the details.

\begin{thm}\label{sec6:thm1}
Let $\mathbf{K}$ be the semi-algebraic set defined in \eqref{eq:bsa_set} and suppose that $\Sigma^p[\bG]$ is Archimedean.
Let $F(\x)$ be a polynomial matrix whose sparsity
graph is chordal and has maximal cliques $\mC_1,\ldots,\mC_t$. If $F(\x)$ is strictly positive definite on $\mathbf{K}$, then there exist SOS matrices $S_{k,i}(\x)$ of size $q_k|\mC_i| \times q_k|\mC_i|$ such that
\begin{equation*}
F(\x)=\sum_{i=1}^tE_{\mC_i}^{\intercal}\left(S_{0,i}(\x)+\sum_{k=1}^m\left\langle S_{k,i}(\x),G_k(\x)\right\rangle_{p}\right)E_{\mC_i}.    
\end{equation*}
\end{thm}
%\begin{proof}
%
%\end{proof}

\begin{thm}\label{sec6:thm2}
Let $\mathbf{K}$ be a semi-algebraic set defined in \eqref{eq:bsa_set} with homogeneous polynomial matrices $G_1,\ldots,G_m$ of even degree and such that $\mathbf{K}\setminus\{\mathbf{0}\}$ is nonempty. Let $F(\x)$ be a homogeneous polynomial matrix of even degree whose sparsity graph is chordal and has maximal cliques $\mC_1,\ldots,\mC_t$. If $F(\x)$ is strictly positive definite on $\mathbf{K}\setminus\{\mathbf{0}\}$, then there exists an integer $\tau\ge0$ and homogeneous SOS matrices $S_{k,i}(\x)$ of size $q_k|\mC_i| \times q_k|\mC_i|$ such that
\begin{equation*}
\|\x\|^{2\tau}F(\x)=\sum_{i=1}^tE_{\mC_i}^{\intercal}\left(S_{0,i}(\x)+\sum_{k=1}^m\left\langle S_{k,i}(\x),G_k(\x)\right\rangle_{p}\right)E_{\mC_i}.    
\end{equation*}
\end{thm}
%\begin{proof}
%
%\end{proof}

\begin{cor}\label{sec6:thm3}
Let $\mathbf{K}$ be a semi-algebraic set defined in \eqref{eq:bsa_set}, and let $F(\x)$ be an inhomogeneous polynomial matrix of even degree whose sparsity graph is chordal and has maximal cliques $\mC_1,\ldots,\mC_t$. If $F(\x)$ is strictly positive definite on $\mathbf{K}$ and its highest-degree homogeneous part is strictly positive definite on $\R^n\setminus\{\mathbf{0}\}$, then there exists an integer $\tau\ge0$ and SOS matrices $S_{k,i}(\x)$ of size $q_k|\mC_i| \times q_k|\mC_i|$ such that
\begin{equation*}
(1+\|\x\|)^{2\tau}F(\x)=\sum_{i=1}^tE_{\mC_i}^{\intercal}\left(S_{0,i}(\x)+\sum_{k=1}^m\left\langle S_{k,i}(\x),G_k(\x)\right\rangle_{p}\right)E_{\mC_i}.    
\end{equation*}
\end{cor}

Then, the decomposed optimization problem for \eqref{eq::mPMI} based on Theorems \ref{sec6:thm1}, \ref{sec6:thm2}, Corollary \ref{sec6:thm3} reads as
\begin{align*}
\sup\lambda\quad\mathrm{s.t.}\quad F(\x)-\lambda&=\sum_{i=1}^tE_{\mC_i}^{\intercal}\left(S_{0,i}(\x)+\sum_{k=1}^m\left\langle S_{k,i}(\x),G_k(\x)\right\rangle_{p}\right)E_{\mC_i},  \\  
\sup\lambda\quad\mathrm{s.t.}\quad \|\x\|^{2\tau}(F(\x)-\lambda)&=\sum_{i=1}^tE_{\mC_i}^{\intercal}\left(S_{0,i}(\x)+\sum_{k=1}^m\left\langle S_{k,i}(\x),G_k(\x)\right\rangle_{p}\right)E_{\mC_i}, \\   
\sup\lambda\quad\mathrm{s.t.}\quad (1+\|\x\|)^{2\tau}(F(\x)-\lambda)&=\sum_{i=1}^tE_{\mC_i}^{\intercal}\left(S_{0,i}(\x)+\sum_{k=1}^m\left\langle S_{k,i}(\x),G_k(\x)\right\rangle_{p}\right)E_{\mC_i},    
\end{align*}
respectively.

\subsection{Constraint matrix sparsity}
We now consider the case that the constraint matrix $G(\x)$ has a chordal sparsity graph $\mathcal{G}$. 
% An example of such a constraint is
% \begin{align}
%     G(\x) = \begin{bmatrix}        1 & x_1 & 0\\
%     x_1 &  1 & x_2\\
%     0 & x_2 & 1\end{bmatrix} \succeq 0,
% \end{align}
% which has two maximal cliques: $\{1, 2\}$ and $
% \{2, 3\}$. 
One way to exploit constraint matrix sparsity is decomposing a single PMI constraint into multiple ones of smaller sizes according to its sparsity pattern (by Theorem \ref{th::spositive2}), which most often brings additional correlative sparsity to exploit. Let us illustrate this method by a concrete example.

\begin{exmp}
Consider the PMO problem
\begin{equation}\label{sec6:ex1}
\inf_{\x\in\R^5}\lambda_{\min}(F(\x))\quad\mathrm{s.t.}\quad G(\x)\succeq 0
\end{equation}
with
\begin{equation}
F(\x) = \begin{bmatrix}x_1 & x_2 & x_3\\
x_2 &  1 & x_4\\
x_3 & x_4 & x_5\end{bmatrix}
\text{ and }\,\,
G(\x) = \begin{bmatrix}1-x_1^2-x_2^2-x_3^2 & x_1x_2x_3 & 0\\
x_1x_2x_3 &  x_3 & x_3x_4x_5\\
0 & x_3x_4x_5 & 1-x_3^2-x_4^2-x_5^2\end{bmatrix}.
\end{equation}
By introducing an auxiliary variable $x_6$ and Theorem \ref{th::spositive2}, we can decompose the PMI constraint $G(\x)\succeq 0$ as
\begin{equation*}
G_1(x_1,x_2,x_3,x_6) = \begin{bmatrix}1-x_1^2-x_2^2-x_3^2 & x_1x_2x_3 \\
x_1x_2x_3 & x_6^2
\end{bmatrix}\succeq 0,
G_2(x_3,x_4,x_5,x_6) =\begin{bmatrix}x_3 -x_6^2& x_3x_4x_5\\
x_3x_4x_5 & 1-x_3^2-x_4^2-x_5^2\end{bmatrix}\succeq 0
\end{equation*}
such that $G\succeq0$ if and only if $G_1\succeq0, G_2\succeq 0$. Then \eqref{sec6:ex1} becomes
\begin{equation}
\inf_{\x\in\R^6}\lambda_{\min}(F(\x))\quad\mathrm{s.t.}\quad G_1(x_1,x_2,x_3,x_6)\succeq 0,\quad G_2(x_3,x_4,x_5,x_6)\succeq 0,
\end{equation}
which exhibits a CSP: $\mI_1=\{1,2,3,6\}$ and $\mI_2=\{3,4,5,6\}$.
\end{exmp}

In the appendix, we provide an alternative way to exploit constraint matrix sparsity.}
\section{Numerical Examples}
\label{sec:examples}

The sparsity-adapted relaxations for PMO problems have been implemented in the Julia package {\tt TSSOS}\footnote{{\tt TSSOS} is freely available at \href{https://github.com/wangjie212/TSSOS}{https://github.com/wangjie212/TSSOS}.} \cite{magron2021tssos}. In this section, we provide numerical examples to illustrate the efficiency of our approach with {\tt TSSOS} 1.4.4 where {\tt Mosek} 11.0 \cite{mosek} is employed as an SDP solver with default settings. When presenting the results in tables, the column labelled by `mb' records the maximal PSD block size of SDP relaxations, the column labelled by `bound' records lower bounds given by SDP relaxations, and the column labelled by `time' records computational time in seconds. Moreover, $r$ stands for the relaxation order, $s$ stands for the sparse order, and the symbol `-' indicates that {\tt Mosek} runs out of memory. For brevity, TS denotes term sparsity, CS denotes correlative sparsity, and MS denotes matrix sparsity.
\rev{All numerical experiments were performed on a desktop computer with an Intel(R) Core(TM) i9-14900 CPU@2.00 GHz and 128G RAM\footnote{The script for reproducing the results is available at \url{https://github.com/wangjie212/TSSOS/blob/master/example/pmi.jl}.}.}

\subsection{Term sparse examples}

Our first example involves the objective matrix 
\begin{equation}
F(\x)=\begin{bmatrix}
x_1^4&x_1^2-x_2x_3&x_3^2-x_4x_5&x_1x_4&x_1x_5\\
x_1^2-x_2x_3&x_2^4&x_2^2-x_3x_4&x_2x_4&x_2x_5\\
x_3^2-x_4x_5&x_2^2-x_3x_4&x_3^4&x_4^2-x_1x_2&x_5^2-x_3x_5\\
x_1x_4&x_2x_4&x_4^2-x_1x_2&x_4^4&x_4^2-x_1x_3\\
x_1x_5&x_2x_5&x_5^2-x_3x_5&x_4^2-x_1x_3&x_5^4\\
\end{bmatrix}
\end{equation}
and constraint matrices 
\begin{equation}
G_1(\x)=\begin{bmatrix}
1-x_1^2-x_2^2&x_2x_3\\
x_2x_3&1-x_3^2
\end{bmatrix},
\quad
G_2(\x)=\begin{bmatrix}
1-x_4^2&x_4x_5\\
x_4x_5&1-x_5^2
\end{bmatrix}.
\end{equation}
The considered optimization problem aims to minimize the minimum eigenvalue of $F$ over the region defined by $G_1,G_2$:
\begin{equation}\label{eq:example_1}
    \inf_{\x \in \R^5} \lambda_{\min}(F(\x))\quad \text{ s.t. } \quad G_1(\x) \succeq 0, \,G_2(\x) \succeq 0.
\end{equation}

Table \ref{tab::1} reports the numerical results of this example, where we compare the dense approach with the sparse approach (exploiting term sparsity with block/chordal closures). It can be seen that except the case of performing chordal closures with $s=2$\footnote{\rev{The chordal closure yields PSD blocks with overlaps. When such overlaps are significant, the computation may get slow.}}, the sparse approach is several times faster than the dense approach while yielding the same lower bound $-2.2766$. The bound $-2.2766$ is certified as the global optimum by checking the flatness condition. For this example, the term sparsity iteration with block closures stabilizes at $s=2$.

\begin{table}[htb]\caption{Results for Problem \eqref{eq:example_1}.}
\label{tab::1}
\centering
% \resizebox{\linewidth}{!}{
\begin{tabular}{ccccccccccccc}
\midrule[0.8pt]
&&\multicolumn{3}{c}{$r=2$}&&\multicolumn{3}{c}{$r=3$}&&\multicolumn{3}{c}{$r=4$}\\
\cline{3-5} \cline{7-9}\cline{11-13}
&&mb& bound & time&&mb&bound & time&&mb&bound & time\\
\midrule[0.4pt]
Dense&&105&-2.2766&0.31&&280&-2.2766&5.34&&630&-2.2766&99.7\\
TS (block, $s=1$)&&50&-2.2766&0.14&&132&-2.2766&1.97&&330&-2.2766&30.3\\
TS (block, $s=2$)&&80&-2.2766&0.15&&200&-2.2766&2.39&&430&-2.2766&35.7\\
TS (chordal, $s=1$)&&48&-2.2766&0.18&&58&-2.2766&0.51&&58&-2.2766&0.93\\
TS (chordal, $s=2$)&&80&-2.2766&0.29&&200&-2.2766&8.05&&430&-2.2766&225\\
\midrule[0.8pt]
\end{tabular}
\end{table}

Next for a positive integer $p\in\{20,40,60\}$, we consider the problem
\begin{equation}\label{eq:example_3}
    \inf_{\x \in \R^3} \lambda_{\min}(F(\x))\quad \text{ s.t. } \quad G(\x) \succeq 0
\end{equation}
with 
\begin{equation}
    F(\x)=(1-x_1^2-x_2^2)I_p+(x_1^2-x_3^2)A+(x_1^2x_3^2-2x_2^2)B,
\end{equation}
and
\begin{equation}
G(\x)=\begin{bmatrix}
1-x_1^2-x_2^2&x_2x_3\\
x_2x_3&1-x_3^2
\end{bmatrix}.
\end{equation}
The matrices $A, B\in \psd^p$ are generated with entries being drawn randomly from the uniform distribution on $(0, 1)$. 
Table \ref{tab::2} reports the numerical results of this problem, where we compare the dense approach with the sparse approach (exploiting term sparsity with block or chordal closures). It can be seen that the sparse approach is significantly faster and more scalable than the dense approach while yielding the same lower bounds. For this problem, the term sparsity iteration with block closures stabilizes at $s=2$.

\begin{table}[htb]\caption{Results for Problem \eqref{eq:example_3} with $r=2$.}
\label{tab::2}
\centering
% \resizebox{\linewidth}{!}{
\begin{tabular}{ccccccccccccc}
\midrule[0.8pt]
&&\multicolumn{3}{c}{$p=20$}&&\multicolumn{3}{c}{$p=40$}&&\multicolumn{3}{c}{$p=60$}\\
\cline{3-5} \cline{7-9}\cline{11-13}
&&mb& bound & time&&mb&bound & time&&mb&bound & time\\
\midrule[0.4pt]
Dense&&200&-20.259&6.13&&400&-39.388&185&&600&-59.560&1756\\
TS (block, $s=1$)&&80&-20.259&0.37&&160&-39.388&7.76&&240&-59.560&55.4\\
TS (block, $s=2$)&&80&-20.259&0.33&&160&-39.388&7.05&&240&-59.560&51.7\\
TS (chordal, $s=1$)&&60&-20.259&0.36&&120&-39.388&8.41&&180&-59.560&68.1\\
TS (chordal, $s=2$)&&80&-20.259&0.37&&160&-39.388&12.6&&240&-59.560&55.1\\
\midrule[0.8pt]
\end{tabular}
\end{table}

\subsection{Correlatively sparse examples}

Consider the problem
\begin{equation}\label{eq:example_6}
    \inf_{\x \in \R^5} \lambda_{\min}(F(\x))\quad \text{ s.t. } \quad G_1(\x) \succeq 0, G_2(\x) \succeq 0
\end{equation}
with
\begin{equation}
F(\x)=\begin{bmatrix}
x_1^4&x_1^2-x_2x_3&x_3^2-x_4x_5&0.5&0.5\\
x_1^2-x_2x_3&x_2^4&x_2^2-x_3x_4&0.5&0.5\\
x_3^2-x_4x_5&x_2^2-x_3x_4&x_3^4&x_4^2-x_1x_2&x_5^2-x_3x_4\\
0.5&0.5&x_4^2-x_1x_2&x_4^4&x_4^2-x_1x_3\\
0.5&0.5&x_5^2-x_3x_4&x_4^2-x_1x_3&x_5^4\\
\end{bmatrix}
\end{equation}
and
\begin{equation}
G_1(\x)=\begin{bmatrix}
1-x_1^2-x_2^2&x_2x_3\\
x_2x_3&1-x_3^2
\end{bmatrix},
\quad
G_2(\x)=\begin{bmatrix}
1-x_4^2&x_4x_5\\
x_4x_5&1-x_5^2
\end{bmatrix},
\end{equation}
which exhibits a CSP: $\mI_1=\{1,2,3\}$, $\mI_2=\{3,4\}$, and $\mI_3=\{4,5\}$.

Table \ref{tab::7} reports the numerical results of this example, where we compare the dense approach with three sparse approaches: exploiting correlative sparsity, exploiting both correlative sparsity and term sparsity with block or chordal closures. It could be seen that the approach exploiting correlative sparsity runs much faster than the dense approach without sacrificing the accuracy, and taking term sparsity into account brings some extra speed-up.

\begin{table}[htb]\caption{Results for Problem \eqref{eq:example_6}.}
\label{tab::7}
\centering
\resizebox{\linewidth}{!}{
\begin{tabular}{ccccccccccccccccc}
\midrule[0.8pt]
\multirow{2}{*}{$r$} &&\multicolumn{3}{c}{Dense}&&\multicolumn{3}{c}{CS}&&\multicolumn{3}{c}{CS + Block}&&\multicolumn{3}{c}{CS + Chordal} \\
\cline{3-5} \cline{7-9}\cline{11-13}\cline{15-17}
&&mb& bound & time&&mb&bound & time&&mb&bound & time&&mb&bound & time \\
\midrule[0.4pt]
2&&105&-2.4131&0.31&&50&-2.4131&0.06&&26&-2.4131&0.03&&12&-2.4148&0.02\\
3&&280&-2.4131&5.27&&100&-2.4131&0.26&&62&-2.4131&0.10&&16&-2.4131&0.04\\
4&&630&-2.4131&121&&200&-2.4131&1.23&&74&-2.4131&0.23&&16&-2.4131&0.08\\
\midrule[0.8pt]
\end{tabular}}
\end{table}

The next example involves a PMO problem with an increasing number of variables. For a given integer $n\ge3$, let 
\begin{equation}
F(\x)=\begin{bmatrix}
\sum_{k=1}^{n-2}x_k^2&\sum_{k=1}^{n-1}x_kx_{k+1}&1\\
\sum_{k=1}^{n-1}x_kx_{k+1}&\sum_{k=2}^{n-1}x_k^2&\sum_{k=1}^{n-2}x_kx_{k+2}\\
1&\sum_{k=1}^{n-2}x_kx_{k+2}&\sum_{k=3}^{n}x_{k}^2
\end{bmatrix}
\end{equation}
and
\begin{equation}
G_k(\x)=\begin{bmatrix}
1-x_k^2-x_{k+1}^2&x_{k+1}+0.5\\
x_{k+1}+0.5&1-x_{k+2}^2
\end{bmatrix},
\quad
k \in[n-2],
\end{equation}
and consider the optimization problem
\begin{equation}\label{eq:example_2}
    \inf_{\x \in \R^n} \lambda_{\min}(F(\x))\quad \text{ s.t. } \quad G_k(\x) \succeq 0,\,k \in[n-2].
\end{equation}

Problem \eqref{eq:example_2} possesses a CSP: $\mI_k=\{k,k+1,k+2\}, k \in[n-2]$.
Table \ref{tab::10} reports the related numerical results with $r=2$ as $n$ increases. As for the previous example, we compare the dense approach with three sparse approaches: exploiting correlative sparsity, exploiting both correlative sparsity and term sparsity with block or chordal closures. We could see that the approach exploiting correlative sparsity yields the same bounds as the dense approach but in much less time (especially when the number of variables $n$ is large). Taking term sparsity into account could further reduce computational time but provides looser bounds.

\begin{table}[htb]\caption{Results for Problem \eqref{eq:example_2} with $r=2$.}
\label{tab::10}
\centering
\begin{tabular}{ccccccccccccccccc}
\midrule[0.8pt]
\multirow{2}{*}{$n$} &&\multicolumn{3}{c}{Dense}&&\multicolumn{3}{c}{CS}&&\multicolumn{3}{c}{CS + Block}&&\multicolumn{3}{c}{CS + Chordal} \\
\cline{3-5} \cline{7-9}\cline{11-13}\cline{15-17}
&&mb& bound & time&&mb&bound & time&&mb&bound & time&&mb&bound & time \\
\midrule[0.4pt]
$5$&&63&-1.0247&0.09&&30&-1.0247&0.03&&16&-1.1931&0.02&&5&-1.5918&0.02\\
$7$&&108&-1.1157&0.38&&30&-1.1157&0.07&&16&-1.3072&0.03&&5&-1.9648&0.03\\
$9$&&165&-1.3891&2.42&&30&-1.3891&0.07&&16&-1.4522&0.04&&5&-2.3454&0.05\\
$11$&&234&-1.7620&9.06&&30&-1.7620&0.09&&16&-1.7620&0.05&&5&-2.7288&0.06\\
$13$&&315&-2.1403&38.2&&30&-2.1403&0.11&&16&-2.1403&0.05&&5&-3.1136&0.08\\
\midrule[0.8pt]
\end{tabular}
\end{table}

\subsection{Matrix sparse examples}
In this subsection, we solve four PMO problems with matrix sparsity.
The first problem is
\begin{equation}\label{eq:example_7}
    \inf_{\x \in \R^5} \lambda_{\min}(F(\x))\quad \text{ s.t. } \quad G_1(\x) \succeq 0, G_2(\x) \succeq 0
\end{equation}
with 
\begin{equation}
F(\x)=\begin{bmatrix}
x_1^4&x_1^2-x_2x_3&x_3^2-x_4x_5&&\\
x_1^2-x_2x_3&x_2^4&x_2^2-x_3x_4&&\\
x_3^2-x_4x_5&x_2^2-x_3x_4&x_3^4&x_4^2-x_1x_2&x_5^2-x_3x_4\\
&&x_4^2-x_1x_2&x_4^4&x_4^2-x_1x_3\\
&&x_5^2-x_3x_4&x_4^2-x_1x_3&x_5^4\\
\end{bmatrix}
\end{equation}
and
\begin{equation}
G_1(\x)=\begin{bmatrix}
1-x_1^2-x_2^2&x_2x_3\\
x_2x_3&1-x_3^2
\end{bmatrix},
\quad
G_2(\x)=\begin{bmatrix}
1-x_4^2&x_4x_5\\
x_4x_5&1-x_5^2
\end{bmatrix}.
\end{equation}

Table \ref{tab::8} reports the related numerical results of Problem \eqref{eq:example_7}, where we compare the dense approach with three sparse approaches: exploiting matrix sparsity, exploiting both matrix sparsity and term sparsity with block or chordal closures. We could see that the three sparse approaches all yield the same bound as the dense approach but in much less time, and exploiting term sparsity brings further speed-up.

\begin{table}[htb]\caption{Results for Problem \eqref{eq:example_7}.}
\label{tab::8}
\centering
\resizebox{\linewidth}{!}{
\begin{tabular}{ccccccccccccccccc}
\midrule[0.8pt]
\multirow{2}{*}{$r$} &&\multicolumn{3}{c}{Dense}&&\multicolumn{3}{c}{MS}&&\multicolumn{3}{c}{MS + Block}&&\multicolumn{3}{c}{MS + Chordal} \\
\cline{3-5} \cline{7-9}\cline{11-13}\cline{15-17}
&&mb& bound & time&&mb&bound & time&&mb&bound & time&&mb&bound & time \\
\midrule[0.4pt]
2&&105&-2.4180&0.28&&63&-2.4180&0.11&&33&-2.4180&0.06&&18&-2.4180&0.05\\
3&&280&-2.4180&5.35&&168&-2.4180&1.71&&90&-2.4180&0.60&&27&-2.4180&0.17\\
4&&630&-2.4180&118&&378&-2.4180&23.8&&198&-2.4180&8.54&&63&-2.4180&0.83\\
\midrule[0.8pt]
\end{tabular}}
\end{table}

Now we consider 
\begin{equation}\label{eq:example_4}
    \inf_{\x \in \R^5} \lambda_{\min}(F(\x))\quad \text{ s.t. } \quad G(\x) \succeq 0
\end{equation}
with
\begin{equation}\label{eq1:example_4}
F(\x)=\begin{bmatrix}
x_1^4+x_2^4+1&x_1x_3\\
x_1x_3&x_3^4+x_4^4+x_5^4+0.5
\end{bmatrix}
\end{equation}
and
\begin{equation}\label{eq2:example_4}
G(\x)=\begin{bmatrix}
1-x_1^2&x_1x_2&x_1x_3&&\\ 
x_1x_2&1-x_2^2&x_2x_3&&\\
x_1x_3&x_2x_3&1-x_3^2&x_3x_4&x_3x_5\\
&&x_3x_4&1-x_4^2&x_4x_5\\ 
&&x_3x_5&x_4x_5&1-x_5^2
\end{bmatrix}.
\end{equation}
By introducing a new variable $x_6$, the PMI constraint $G(\x)\succeq0$ can be decomposed as 
\begin{equation*}
G_1(x_1,x_2,x_3,x_6)=\begin{bmatrix}
1-x_1^2&x_1x_2&x_1x_3\\ 
x_1x_2&1-x_2^2&x_2x_3\\
x_1x_3&x_2x_3&x_6^2
\end{bmatrix}\succeq0,
G_2(x_3,x_4,x_5,x_6)=\begin{bmatrix}
1-x_3^2-x_6^2&x_3x_4&x_3x_5\\
x_3x_4&1-x_4^2&x_4x_5\\ 
x_3x_5&x_4x_5&1-x_5^2
\end{bmatrix}\succeq0
\end{equation*}
such that $G\succeq0$ if and only if $G_1\succeq 0, G_2\succeq 0$. 
Note that after the decomposition, the PMO problem exhibits a CSP: $\mI_1=\{1,2,3,6\}, \mI_2=\{3,4,5,6\}$.

Table \ref{tab::4} reports the related numerical results of Problem \eqref{eq:example_4}, where we compare the dense approach with three sparse approaches: exploiting both matrix sparsity and correlative sparsity, exploiting matrix sparsity, correlative sparsity, and term sparsity with block or chordal closures. It could be seen that the three sparse approaches all yield the same bound as the dense approach while in much less time (especially when $r\ge3$), and exploiting term sparsity brings extra speed-up.

\begin{table}[htb]\caption{Results for Problem \eqref{eq:example_4}.}
\label{tab::4}
\centering
\begin{tabular}{ccccccccccccccccc} 
\midrule[0.8pt]
\multirow{2}{*}{$r$} &&\multicolumn{3}{c}{Dense}&&\multicolumn{3}{c}{MS+CS}&&\multicolumn{3}{c}{MS+CS+Block}&&\multicolumn{3}{c}{MS+CS+Chordal} \\
\cline{3-5} \cline{7-9}\cline{11-13}\cline{15-17}
&&mb& bound & time&&mb&bound & time&&mb&bound & time&&mb&bound & time \\
\midrule[0.4pt]
$2$&&60&0.3977&0.05&&30&0.3977&0.04&&6&0.3977&0.01&&5&0.3977&0.01\\
$3$&&210&0.3977&0.82&&90&0.3977&0.17&&14&0.3977&0.02&&7&0.3977&0.02\\
$4$&&560&0.3977&4.94&&210&0.3977&1.55&&26&0.3977&0.05&&15&0.3977&0.05\\
\midrule[0.8pt]
\end{tabular}
\end{table}

We next consider a PMO problem adapted from \cite[Sec.~6]{kim2011exploiting}:
\begin{equation}\label{eq:example_9}
    \inf_{\x \in \R^n} \lambda_{\min}(F(\x))\quad \text{ s.t. } \quad G(\x) \succeq 0
\end{equation}
with
\begin{equation}
    F(\x)=\begin{bmatrix}
        1 &x_1x_2\\
        x_1x_2 &1+x_n^2
    \end{bmatrix}
\end{equation}
and
\begin{equation}
    G(\x)=\begin{bmatrix}
        1-x_1^4&0&0&\cdots&0&x_1x_2\\
        0&1-x_2^4&0&\cdots&0&x_2x_3\\
        0&0&\rev{1-x_3^4}&\cdots&0&x_3x_4\\
        \vdots&\vdots&\vdots&\ddots&\vdots&\vdots\\
        0&0&0&\cdots&1-x_{n-1}^4&x_{n-1}x_n\\
        x_1x_2&x_2x_3&x_3x_4&\cdots&x_{n-1}x_n&1-x_{n}^4\\
    \end{bmatrix}\succeq0.
\end{equation}
By introducing $n-2$ new variables $\{x_{n+1},\ldots,x_{2n-2}\}$, the PMI constraint $G(\x)\succeq0$ can be decomposed as 
\begin{equation*}
G_1(x_1,x_2,x_{n+1})=\begin{bmatrix}
1-x_1^4&x_1x_{2}\\ 
x_1x_{2}&x_{n+1}^2
\end{bmatrix}\succeq0,\quad
G_{n-1}(x_{n-1},x_n,x_{2n-2})=\begin{bmatrix}
1-x_{n-1}^4&x_{n-1}x_n\\ 
x_{n-1}x_n&1-x_n^4-x_{2n-2}^2
\end{bmatrix}\succeq0,
\end{equation*}
and
\begin{equation*}
G_i(x_i,x_{i+1},x_{n+i-1},x_{n+i})=\begin{bmatrix}
1-x_i^4&x_ix_{i+1}\\ 
x_ix_{i+1}&x_{n+i}^2-x_{n+i-1}^2
\end{bmatrix}\succeq0,\quad i=2,\ldots,n-2,
\end{equation*}
such that $G\succeq 0$ if and only if $G_{i}\succeq 0, i\in[n-1]$.
%G=\sum_{i=1}^{n-1}G_{i}$.
Note that after the decomposition, the PMO problem exhibits a CSP: $\mI_1=\{1,2,n+1\},\mI_i=\{i,i+1,n+i-1,n+i\} (i=2,\ldots,n-2), \mI_{n-1}=\{n-1,n,2n-2\}$.

Table \ref{tab::5} reports the related numerical results of Problem \eqref{eq:example_9} with $n\in\{5,7,9,11,13\}$, where we compare the dense approach with three sparse approaches: exploiting both matrix sparsity and correlative sparsity, exploiting matrix sparsity, correlative sparsity, and term sparsity with block or chordal closures. One could see that the three sparse approaches all yield the same bound as the dense approach while in much less time when $n\in\{5,7\}$, and scale much better with large $n$.

\begin{table}[htb]\caption{Results for Problem \eqref{eq:example_9} with $r=4$.}
\label{tab::5}
\centering
\begin{tabular}{ccccccccccccccccc}
\midrule[0.8pt]
\multirow{2}{*}{$n$} &&\multicolumn{3}{c}{Dense}&&\multicolumn{3}{c}{MS+CS}&&\multicolumn{3}{c}{MS+CS+Block}&&\multicolumn{3}{c}{MS+CS+Chordal} \\
\cline{3-5} \cline{7-9}\cline{11-13}\cline{15-17}
&&mb& bound & time&&mb&bound & time&&mb&bound & time&&mb&bound & time \\
\midrule[0.4pt]
$5$&&252&0.2138&2.81&&140&0.2138&0.80&&15&0.2138&0.03&&15&0.2138&0.03\\
$7$&&660&0.2138&138&&140&0.2138&1.23&&15&0.2138&0.05&&15&0.2138&0.05\\
$9$&&1430&-&-&&140&0.2138&2.15&&15&0.2138&0.07&&15&0.2138&0.07\\
$11$&&2730&-&-&&140&0.2138&2.64&&15&0.2138&0.09&&15&0.2138&0.09\\
$13$&&4760&-&-&&140&0.2138&3.24&&15&0.2138&0.10&&15&0.2138&0.10\\
\midrule[0.8pt]
\end{tabular}
\end{table}

Let $\omega$ be an even positive integer and consider the following optimization problem with a PMI constraint:
\begin{equation}\label{example_5}
\begin{cases}
\inf\limits_{\lambda_1,\lambda_2} &\lambda_2-10\lambda_1\\
\,\,\mathrm{s.t.}&P_{\omega}(\x)=\begin{bmatrix}
        \lambda_2x_1^4+x_2^4&\lambda_1x_1^2x_2^2&&&&&\\
        \lambda_1x_1^2x_2^2&\lambda_2x_2^4+x_3^4&\lambda_2x_2^2x_3^2&&&&\\
        &\lambda_2x_2^2x_3^2&\lambda_2x_3^4+x_1^4&\lambda_1x_1^2x_3^2&&&\\
        &&\lambda_1x_1^2x_3^2&\lambda_2x_1^4+x_2^4&\lambda_2x_1^2x_2^2&&\\
        &&&\lambda_2x_1^2x_2^2&\lambda_2x_2^4+x_3^4&\ddots&\\
        &&&&\ddots&\ddots&\lambda_1x_2^2x_3^2\\
        &&&&&\lambda_1x_2^2x_3^2&\lambda_2x_3^4+x_1^4\\
    \end{bmatrix}\succeq0,
\end{cases}
\end{equation}
where $P_{\omega}(\x)$ is a $3\omega\times3\omega$ tridiagonal polynomial matrix \cite{zheng2023sum}. The sparsity graph of $P_{\omega}(\x)$ has maximal cliques: $\mC_i=\{i,i+1\}, i\in[3\omega-1]$. By Theorem \ref{sec6:thm2}, Problem \eqref{example_5} can be reformulated as
\begin{equation}\label{zheng}
    \begin{cases}
\inf\limits_{\lambda_1,\lambda_2} &\lambda_2-10\lambda_1\\
\,\,\mathrm{s.t.}&\|\x\|^{2\tau}P_{\omega}(\x)=\sum_{i=1}^{3\omega-1}E_{\mC_i}^{\intercal}S_{i}(\x)E_{\mC_i},\\
&S_{i}(\x)\text{ is an SOS matrix}, \quad \forall i\in[3\omega-1].
\end{cases}
\end{equation}

Table \ref{tab::6} reports the related numerical results of Problem \eqref{example_5} with $\tau=3, r=5$ and $\omega\in\{10,20,30,40,50\}$, where we compare two sparse approaches: exploiting matrix sparsity and exploiting both matrix sparsity and term sparsity with block closures. We could see that the two sparse approaches always yield the same bounds and exploiting term sparsity further speeds up the computation by several times.

\begin{table}[htb]\caption{Results for Problem \eqref{example_5} with $\tau=3, r=5$.}
\label{tab::6}
\centering
\begin{tabular}{ccccccccc}
\midrule[0.8pt]
\multirow{2}{*}{$\omega$} &&\multicolumn{3}{c}{MS}&&\multicolumn{3}{c}{MS+TS} \\
\cline{3-5} \cline{7-9}
&&mb& bound & time&&mb&bound & time \\
\midrule[0.4pt]
$10$&&112&-9.0933&4.01&&20&-9.0933&0.47\\
$20$&&112&-9.0240&8.32&&20&-9.0240&0.87\\
$30$&&112&-9.0108&11.2&&20&-9.0108&1.36\\
$40$&&112&-9.0061&13.3&&20&-9.0061&1.70\\
$50$&&112&-9.0039&16.4&&20&-9.0039&2.13\\
\midrule[0.8pt]
\end{tabular}
\end{table}

\rev{
\subsection{Randomly generated sparse examples}
Finally, we test randomly generated sparse examples. Given $t\in\N\setminus\{0\}$, let $\mI_i=\{3i-2,\ldots,3i+2\}$ and let $F_i(\x(\mI_i)),G_i(\x(\mI_i))\in\mS^4[\x(\mI_i)]$ be random polynomial matrices whose entries are sparse quadratic polynomials with coefficients randomly chosen according to the standard normal distribution for $i=1,\ldots,t$. Then we consider the following PMO problem:
\begin{equation}\label{example_10}
\begin{cases}
\inf\limits_{\x \in \R^{3t+2}} &\lambda_{\min}\left(\sum_{i=1}^tF_i(\x(\mI_i))\right)\\
\quad\mathrm{s.t.}&G_i(\x(\mI_i)) + 2\sum_{j\in\mI_i}x_j^2\cdot I_4\succeq 0,\quad i=1,\ldots,t,\\
&1- \sum_{j\in\mI_i}x_j^2\ge 0,\quad i=1,\ldots,t.
\end{cases}
\end{equation}

Table \ref{tab::9} reports the related numerical results of Problem \eqref{example_10} with $t\in\{2,4,6,8,10\}$, where we compare the dense approach with three sparse approaches: exploiting correlative sparsity, exploiting both correlative sparsity and term sparsity with block or chordal closures. One could see that the three sparse approaches run much faster and scale significantly better than the dense approach while providing slightly weaker bounds.

\begin{table}[htb]\caption{Results for Problem \eqref{example_10} with $r=2$.}
\label{tab::9}
\centering
\begin{tabular}{ccccccccccccccccc}
\midrule[0.8pt]
\multirow{2}{*}{$t$} &&\multicolumn{3}{c}{Dense}&&\multicolumn{3}{c}{CS}&&\multicolumn{3}{c}{CS+block}&&\multicolumn{3}{c}{CS+chordal} \\
\cline{3-5} \cline{7-9}\cline{11-13}\cline{15-17}
&&mb& bound & time&&mb&bound & time&&mb&bound & time&&mb&bound & time \\
\midrule[0.4pt]
$2$&&180&-2.4648&2.86&&96&-2.4648&0.41&&27&-2.5195&0.07&&7&-2.5722&0.06\\
% $3$&&312&-3.3027&436&&96&-3.3317&2.39&&80&-3.3435&0.59\\
$4$&&480&-3.5626&365&&96&-3.7565&0.89&&80&-3.8478&0.41&&14&-4.1282&0.20\\
% $5$&&112&-9.0061&18.4&&&&&&20&-9.0061&4.27\\
$6$&&924&-&-&&96&-4.8454&1.50&&80&-4.9889&0.57&&8&-5.5581&0.19\\
$8$&&1512&-&-&&96&-6.5244&2.07&&80&-6.5449&0.88&&14&-7.0377&0.40\\
$10$&&2244&-&-&&96&-7.8778&3.18&&80&-7.8858&1.16&&13&-8.2506&0.46\\
\midrule[0.8pt]
\end{tabular}
\end{table}}

\section{Conclusions}
\label{sec:conclusion}
% SDP-based hierarchies for polynomial optimization problems with \ac{PMI} constraints scale in a jointly polynomial manner with increasing number of variables $n$ and degree $d$. Problem structure can be utilized if present to reduce the impact of this curse of dimensionality. 
This paper \rev{has explored} the use of various sparsity methods in reducing the size of matrix Moment-SOS relaxations for verification of \acp{PMI} and solving \ac{PMO} problems. We have showed that multiple sparsity structures, \rev{e.g., term sparsity, correlative sparsity, and matrix sparsity,} could be simultaneously exploited to maximize \rev{the dimensionality reduction of the derived \ac{SDP} in the PMO setting}. These methods make the matrix Moment-SOS hierarchy more capable of tackling practical applications, as illustrated \rev{by extensive numerical experiments in which structured PMO problems can be tractably solved whereas the dense hierarchy mostly often runs out of memory.}
% decompositions of \ac{PMI} impositions in the unconstrained and constrained settings. Term sparsity was exploited through the use of two methods: decreasing the number of monomials using Newton Polytope techniques, and reducing the size of any PSD matrix in a Gram constraint using (sign) symmetry reduction. Each of these methods were originally introduced in the scalar polynomial context, and were adapted in this paper to the matrix case. In particular, we introduced a notion of a matrix term sparisty pattern that incorporates the both the monomial structure (multiplications) and the position of present monomials in the matrix.
The sparsity routines introduced in this paper were incorporated into the open-source {\tt TSSOS} package, and are thus available to interested practitioners in fields such as optimization, control, and operations research. 
% Future research includes improving the numerical performance of SDP-based solution methods for polynomial optimization problems (w.r.t. time and numerical conditioning), investigating new decomposable structures to simplify computation, and quantifying the conservatism introduced in chordal closure schemes.

% An extended Arxiv version of this paper is available at \urg{[Arxiv link goes here] what will the arxiv version have that this work does not?
% }.

\section*{Acknowledgements}
We sincerely thank the Associate Editor and reviewers for their insightful comments, which have significantly improved the manuscript.
The authors thank Shenyuan Ma, Yang Zheng,  Mareike Dressler, and Timo de Wolff for discussions about term sparsity patterns in \acp{PMI}, and thank Roy S. Smith for his advice, support, and funding.
\bibliographystyle{IEEEtran}
\bibliography{references,spmi}

\appendix
\newpage

\section{An alternative exploitation of constraint matrix sparsity}\label{app1}
Here, we provide another way to exploit constraint matrix sparsity.
We first extend Theorem \ref{th::spositive} towards block-partitioned matrices.
Suppose that $A\in\R^{pq\times pq}$ is a block matrix of form
\[A=\left[
		\begin{array}{cccc}
			A_{11}	& A_{12} & \cdots & A_{1p}\\
			A_{21}	& A_{22} & \cdots & A_{2p}\\
			\vdots & \vdots & \ddots & \vdots\\
			A_{p1}	& A_{p2} & \cdots & A_{pp}
		\end{array}
	\right],
\]
where each block $A_{ij}\in\R^{q\times q}$,
$i,j=1,\ldots,p$. 
Given an undirected graph $\mG(\vs,\mE)$ with nodes
$\vs=\{1,\ldots,p\}$,
we define the set of $q$-partitioned sparse symmetric matrices as
\[
	\psd^{pq}_{q}(\mG,0)\coloneqq\{A\in\psd^{pq} \mid
	A_{ij}=A_{ji}^{\intercal}=0, \text{ if } i\neq j  \text{ and } \{i,j\}\not\in\mE\},
\] 
and define the cone of $q$-partitioned completable \ac{PSD} matrices as
\begin{align*}
\psd_{q,+}^{pq}(\mG,?)\coloneqq\Pi^q_{\mG}(\mS_+^{pq})=\{\Pi^q_{\mG}(A)\mid A\in\mS_+^{pq}\},
\end{align*}
where $\Pi^q_{\mG}:\psd^{pq}\rightarrow\psd^{pq}_{q}(\mG,0)$ is the projection given by 
\begin{equation}
[\Pi_{\mG}(A)]_{ij}=\begin{cases}
A_{ij}, &\textrm{if }i=j\textrm{ or }\{i,j\}\in \mE,\\
0, &\textrm{otherwise}.
\end{cases}
\end{equation}
Given any maximal clique $\mC_i$ of $\mG$, define the block-wise index matrix 
$E_{\mC_i,q}\in\R^{|\mC_i|q\times pq}$ by 
\[
	[E_{\mC_i,q}]_{jk}=\left\{\begin{array}{ll}
		I_{q}, &\text{if }	\mC_i(j)=k,\\
		0, &\text{otherwise}.
	\end{array}\right.
\]
The following theorem extends Theorem
\ref{th::spositive} to the case of block matrices.
% \begin{thm}[Block-chordal decomposition theorem]\label{th::bspositive2}
% 	Let $\mG(\vs,\mE)$ be a chordal graph and let
% 	$\{\mC_1,\ldots,\mC_t\}$ be the set of its maximal cliques. Given
% 	a partition $\gamma=(\gamma_1,\ldots,\gamma_{m})$ with
% 	$N=\sum_{j=1}^p\gamma_j$, then $P\in\psd^N_{\gamma,
% 	+}(\mE,0)$ if and only if there are matrices
% 	$P_i\in\psd_+^{|\mC_i|_{\gamma}}$ for $i=1,\ldots,t$, such
% 	that
% 	$P=\sum_{i=1}^{\intercal}E_{\mC_i,\gamma}^{\intercal}P_iE_{\mC_i,\gamma}$.
% \end{thm}
\begin{thm}{\upshape(\cite[Theorem 2.18]{zheng2019chordal})}\label{th::bspositive}
Let $\mG(\vs,\mE)$ be a chordal graph and let
$\{\mC_1,\ldots,\mC_t\}$ be the set of its maximal cliques. Then $A\in\psd^{pq}_{q, +}(\mG, ?)$ if and only if $E_{\mC_i,q}AE^{\intercal}_{\mC_i,q}\in\psd_+^{|\mC_i|q}$ for all $i\in[t]$. 
\end{thm}

The following proposition allows us to consider an appropriate PSD completable matrix for representing $\langle S(\x), G(\x)\rangle_p$ when $G(\x)$ is sparse.

% Denote the degree $d_G$ as
% \[
% 	d_G:=\max\{\lceil\deg(G_{i,j})/2\rceil \mid i, j=1,\ldots,m\}.
% \]
% For any $k\in\N$, recall that $s_{k-d_G}=|\N^n_{k-d_G}|$ and let
% $\gamma(k)=(s_{k-d_G},\ldots,s_{k-d_G})\in\mathbb{Z}^{m}$ be the uniform partion.
\begin{prop}
Let $G(\x)\in\mS^q[\x]$ be a polynomial matrix that has a chordal sparsity graph $\mathcal{G}$, and let $S(\x)\in\mS^{pq}[\x]$ be an SOS matrix of degree $2d$. Then there exists $Q\in\mS^{pq|\bm_{d}(\x)|}_{p|\bm_{d}(\x)|,+}(\G,?)$
such that
\begin{equation}\label{eq::chordalSOS}
\langle S(\x), G(\x)\rangle_p =\left\langle (I_{pq}\otimes \bm_{d}(\x))^{\intercal} \tilde{Q} (I_{pq} \otimes \bm_{d}(\x)), G(\x)\right\rangle_p,
\end{equation}
% =\left\langle Q, G(x)\otimes (\bm_{d}(\x)\cdot\bm_{d}(\x)^{\intercal})
% \right\rangle_p.
where $\tilde{Q}$ is obtained from $Q$ by certain row and column permutations.
\end{prop}
\begin{proof}
By \eqref{eq:sos_pmi_uncons2}, there exists $Q\in\mS_+^{pq|\bm_{d}(\x)|}$ such that 
\[
S(\x)=(I_{pq}\otimes \bm_{d}(\x))^{\intercal} Q (I_{pq} \otimes \bm_{d}(\x)). 
\]
Let $Q=[Q_{ij}]_{i,j\in[p]}$ with blocks 
$Q_{ij}=Q_{ji}^\intercal\in\R^{q|\bm_{d}(\x)|\times q|\bm_{d}(\x)|}$. 
By definition, we have
\begin{align*}
    \langle S(\x), G(\x)\rangle_p&=\begin{bmatrix}
        \langle S_{11}(\x), G(\x)\rangle&\ldots&\langle S_{1p}(\x), G(\x)\rangle\\
        \vdots&\ddots&\vdots\\
        \langle S_{p1}(\x), G(\x)\rangle&\ldots&\langle S_{pp}(\x), G(\x)\rangle\\
    \end{bmatrix}\\
    &=\begin{bmatrix}
        \langle G(x)\otimes(\bm_{d}(\x)\cdot\bm_{d}(\x)^{\intercal}), Q_{11}\rangle&\ldots&\langle G(x)\otimes(\bm_{d}(\x)\cdot\bm_{d}(\x)^{\intercal}), Q_{1p}\rangle\\
        \vdots&\ddots&\vdots\\
        \langle G(x)\otimes(\bm_{d}(\x)\cdot\bm_{d}(\x)^{\intercal}), Q_{p1}\rangle&\ldots&\langle G(x)\otimes(\bm_{d}(\x)\cdot\bm_{d}(\x)^{\intercal}), Q_{pp}\rangle\\
    \end{bmatrix}.
\end{align*}
%for some $\hat{Q}_{i,j}\in\mS^{q|\bm_{d}(\x)|}_{+},i,j\in[p]$. 
Note that $G(\x)\otimes(\bm_{d}(\x)\cdot\bm_{d}(\x)^{\intercal})$ has a block sparse structure 
induced by the chordal sparsity of $G(\x)$. 
Therefore, up to certain row and column permutations, we may assume that $Q\in\mS^{pq|\bm_{d}(\x)|}_{p|\bm_{d}(\x)|,+}(\G,?)$.
%$\hat{Q}_{i,j}\in\mS^{q|\bm_{d}(\x)|}_{|\bm_{d}(\x)|,+}(\G,?),i,j\in[p]$. 
%Let $\hat{Q}$ be the block matrix with blocks $\hat{Q}_{i,j}$, so that
%\begin{equation*}
%    \left\langle S(\x), G(\x)\right\rangle_p=\left\langle \hat{Q}, G(x)\otimes(\bm_{d}(\x)\cdot\bm_{d}(\x)^{\intercal})\right\rangle_p=\left\langle (\bm_{d}(\x)\otimes I_{pq})^{\intercal} \tilde{Q} (\bm_{d}(\x)\otimes I_{pq}), G(\x)\right\rangle_p,
%\end{equation*}
%where 
% Clearly, the matrix $\tilde{Q}$ satisfying \eqref{eq::chordalSOS} can be obtained from $Q$ 
% by certain row and column permutations. The desired conclusion then follows.
\end{proof}

\end{document}